\documentclass[a4paper]{article}

\usepackage[utf8]{inputenc}  
\usepackage{a4wide}
\usepackage{amsmath,amsthm,amssymb}
\usepackage{mathtools}       
\usepackage{graphicx}
\usepackage{authblk}
    \graphicspath{{./}{figures/}}
\usepackage{bm}              
\usepackage{booktabs} 
\usepackage[UKenglish]{babel}
\usepackage{xcolor} 
\usepackage[colorlinks=true,linkcolor=blue,citecolor=red]{hyperref}
\usepackage{cite}
\usepackage[shortlabels]{enumitem}
\usepackage[left=1in,right=1in]{geometry}

\newcommand*{\C}{\mathcal C}

\newcommand*{\R}{\mathbb R}
\let\I\C
\newcommand*{\Rp}{\R^+}
\newcommand*{\finf}[1][\delta]{f_{J, #1}^\infty}
\newcommand*{\finfU}[1][\delta]{{\vphantom{f}}_U\kern-.05em f_{J, #1}^\infty}
\newcommand*{\intRp}{\int_{\Rp}}
\newcommand*{\z}{\mathbf{z}}
\newcommand*\pd[2]{\frac{\partial #1}{\partial #2}}

\renewcommand*{\*}{_\ast}
\renewcommand{\(}{\begin{equation}}
\renewcommand{\)}{\end{equation}}
\DeclareMathOperator{\sign}{sign}
\DeclareMathOperator{\signep}{sign_\epsilon}
\DeclareMathOperator{\sed}{sign^\prime_\epsilon}

\newif\ifbalancedquotes 
\global\balancedquotestrue 
\def\virgolette{\ifbalancedquotes 
                    \char96\char96 
                    \global\balancedquotesfalse 
                \else
                   \char39\char39 
                  \global\balancedquotestrue 
\fi}
\newif\ifautoquote 
\newif\ifactivequote 
\if\catcode`"=\active%
    \activequotetrue
\else
    \activequotefalse
\fi
\ifactivequote 
    \autoquotefalse\relax
\else
    \global\autoquotetrue%
\fi
\ifautoquote 
    \catcode`"=\active 
    \let"=\virgolette 
\fi

\let\epsilon=\varepsilon
\let\phi=\varphi

\DeclarePairedDelimiter{\lrp}{(}{)} 
\DeclarePairedDelimiter{\set}{\lbrace}{\rbrace}         
\DeclarePairedDelimiter{\abs}{\lvert}{\rvert}           
\DeclarePairedDelimiter{\bracket}{\langle}{\rangle}     
\DeclarePairedDelimiter{\norma}{\lVert}{\rVert}          

\newtheorem{theorem}{Theorem}[section]
\newtheorem{proposition}[theorem]{Proposition}

\newtheorem{cor}[theorem]{Corollary}
\theoremstyle{remark}\newtheorem{remark}[theorem]{Remark}

\colorlet{corr}{magenta}
\def\reply#1{\textcolor{corr}{#1}}
\def\replydec{\color{corr}}
\colorlet{corr}{black}

\begin{document}
\title{On the optimal control of kinetic epidemic models with uncertain social 
features}

\author[1]{J. Franceschi\thanks{\tt jonathan.franceschi01@universitadipavia.it 
}}
\author[1]{A. Medaglia\thanks{\tt andrea.medaglia02@universitadipavia.it }}
\author[1]{M. Zanella\thanks{\tt mattia.zanella@unipv.it}}
		
\affil[1]{Department of Mathematics ``F. Casorati'', University of Pavia, Italy}
\date{}

\maketitle

\begin{abstract}
It is recognized that social heterogeneities in terms of the contact distribution have a strong influence on the spread of infectious diseases. Nevertheless, few data are available \reply{on the group composition of social contacts,} and  their statistical description does not possess universal patterns and may vary spatially and temporally. It is therefore essential to design \reply{robust} control strategies, mimicking the effects of non-pharmaceutical interventions, to limit efficiently the number of infected cases. In this work, starting from 
a recently introduced kinetic model for epidemiological dynamics that takes 
into account the impact of social contacts of individuals, we consider an 
uncertain contact formation dynamics leading to slim-tailed as well as 
fat-tailed distributions of contacts. Hence, we analyse the effects of an \reply{optimally robust} control strategy of the system of agents. Thanks to classical methods of kinetic theory, we couple uncertainty quantification methods with the introduced mathematical model to assess the effects of social limitations. Finally, using the proposed modelling approach and starting from available data, we show the effectiveness of the proposed selective measures to dampen uncertainties together with the epidemic trends.

\medskip


\end{abstract}

\tableofcontents
\section{Introduction}

In recent years extensive research efforts have been devoted to design 
effective non-pharmaceutical interventions (NPIs) to mitigate the impact of the 
COVID-19 pandemics \cite{APZ21_a,Gatto20,Gatto2,McQuade21,Imperial,pellis}. In 
particular, several works in mathematical epidemiology shed light on the 
importance of the inner heterogeneity in the social structure of a population, 
see \cite{APZ21_b,Dimarco21,Dolbeault21,Zhang20}. In this direction, among the main factors shaping the evolution of the epidemic, the contact structure of a 
population has been deeply studied especially in relation to the age 
distribution of a population. Special attention was recently paid by the 
scientific community to the role and the estimate of the distribution of 
contacts between individuals as also a relevant cause of the potential pathogen 
transmission \cite{Beraud15,fumanelli12,britton2020}. Nevertheless, we have 
often limited information on the real social features of a population, whose 
characteristics are structurally uncertain and may frequently change due to 
exogenous processes that are also influenced by psychological factors, 
determining different responses in terms of individuals' protective behavior, 
see e.g. \cite{durham12,giambiagi21}. 

Starting from the above considerations, recent works proposed kinetic-type 
models to connect the distribution of social contacts with the spreading of a 
disease in multi-agent systems \cite{DPTZ20,Dimarco21,Medaglia21,Zanella21}. 
The result is obtained by integrating a compartmental modeling approach for 
epidemiological dynamics with a thermalization process determining the 
formation of social contacts. We highlight how the advantages of kinetic 
modeling approaches for epidemiological dynamics rely on a clear connection 
between the scales of the transmission of the infection, linking agent-based 
dynamics with the macroscopic observable ones. Within this research framework, 
we mention \cite{LT,DMLT} where epidemiological relevant states are 
characterized by agent-based viral load dynamics. 

In this paper, we concentrate on a classical SEIR compartmentalization of the 
population whose contact distribution is uncertain. In particular, we introduce 
an interaction scheme describing the evolution in the number of social contacts 
of individuals. The microscopic model is based on a simple transition operator 
whose parameters are assumed to be uncertain. At the kinetic level, the 
aforementioned model is capable to identify a variety of equilibrium 
distributions, ranging from slim-tailed Gamma-type distributions to 
power-law-type distributions depending on the introduced uncertainties. \reply{In the introduced setting, the analysis of the emerging distribution is essential to define the evolution of the main moments of the system of kinetic equations via a closure approach determining the evolution of macroscopic quantities. In particular, we will consider stationary states that depend on uncertain quantities thus, the derived   
system of equations embeds an incomplete knowledge on the real distribution of contacts.}

Therefore, the definition of effective NPIs, generally based on a generalized  reduction of the number of contacts, should take into account the uncertain   contact structure of a population. In particular, we aim at giving a deeper  understanding of the mitigation effects due to the reduction of social  interactions among individuals. To this end, we develop an approach sufficiently robust in terms of the introduced uncertainties. This is done through a combination of a kinetic epidemiological model and a control strategy whose target is to point the population towards a given target number of contacts. The development of  control protocols for kinetic and mean-field equations has been deeply investigated in recent years, without pretending to review the huge literature we mention \cite{Albi14,albi18,ACFK,FPR14,PDT19} and the references therein.  
In detail, we concentrate on modeling the lockdown policies through a 
selective optimal control approach. In particular, we show how the form of the implemented control may result in very different mitigation effects, that deeply depend on the heterogeneity in the contact distribution of the population. In the last part, starting from the calibrated model at our disposal, we focus on the numerical study of the proposed approach and we exploit accurate methods for the uncertainty quantification of kinetic equations. 

The rest of the paper is organized as follows. In Section \ref{sect:2} we 
introduce a system of kinetic equations with SEIR compartmentalization 
combining the dynamics of social contacts with the spread of an infectious 
disease in a multi-agent system. The main features of the solution of a 
surrogate Fokker-Planck model are studied in Section \ref{sec:uniqueness}. In 
Section \ref{sect:3} a control strategy is introduced at the kinetic level and 
in Section \ref{sect:4} we observe the effects of the control on the 
corresponding second-order macroscopic model. Finally, in Section 
\ref{sec:numerics} we investigate numerically the relationship between the 
kinetic epidemic model with uncertainties and its macroscopic limit. A second 
part is dedicated to the interface between the introduced modeling approach and 
available data.

\section{Kinetic epidemic models with uncertain contact 
distribution}\label{sect:2}
In this section, we introduce a compartmental model describing the spreading of 
an infectious disease coupled with a kinetic-type description of the contact 
evolution of a system of individuals \cite{Dimarco22, Dimarco21, Zanella21, 
Medaglia21}. In addition, we will also take into account uncertainties 
collecting the missing information on the contact distribution.

In more details, we consider a system of agents that can be subdivided into the 
following relevant epidemiological states \cite{Hethcote00, Brauer19, 
Diekmann00}: susceptible (S) agents are the ones that can contract the disease, 
infectious agents (I) are responsible for the spread of the disease, exposed 
(E) agents have been in contact with infectious ones but still may or may not 
become contagious; finally, removed (R) agents cannot spread the disease. 

To incorporate the impact of contact distribution in the infectious dynamics, 
we denote by $f_J = f_J(\z,x,t)$ the distribution of \reply{the number of contacts $x\in \mathbb R^+$} at time $t\ge 0$ 
of agents in compartment^^>$J$, where $J \in \C \coloneqq \set{S, E, I, R}$. 
The random vector $\z \in I_z \subseteq \mathbb{R}^{d_{\z}}$, with $d_{\z} \in 
\mathbb{N}$, collects all the uncertainties of the system and we suppose to 
know its distribution $p(\z)$ such that
\[
\textrm{Prob}(\z\in I_z)=\int_{I_z} p(\z) d \z.
\]
We define the total contact distribution of a society as
\[
\sum_{J \in \C} f_J(\z,x,t) = f(\z,x,t),\qquad \intRp f(\z,x,t) \, dx = 1,
\]
while the mass fractions of the population in each compartment and their moment 
of order $r>0$ are given by 
\[
\rho_J(\z,t) = \intRp f_J(\z,x,t) \, dx, \qquad \rho_J(\z,t)m_{r,J}(\z,t) = 
\intRp x^r f_J(\z,x,t)dx. 
\]
In the following, to simplify notations we will indicate with $m_J(\z,t)$, $J 
\in \C$, the mean values corresponding to $r =1$.

Hence, we assume that the introduced compartments in the model could act 
differently at the level of the social process constituting the contact 
dynamics. The kinetic model defining the time evolution of the 
functions^^>$f_J(\z, x, t)$ follows by combining the epidemic process with the 
contact dynamics. This gives the system
\begin{equation}
\left\lbrace
\begin{aligned}
\pd{f_S(\z,x,t)}{t}  &= -K(f_S, f_I)(\z,x,t) 
                     + \frac{1}{\tau} Q_S(f_S)(\z,x,t),\\ 
\pd{f_E(\z,x,t)}{t}  &= K(f_S, f_I)(\z,x,t) - \reply{\zeta} f_E(\z,x,t)
                     + \frac{1}{\tau} Q_E(f_E)(\z,x,t),\\
\pd{f_I(\z,x,t)}{t}  &= \reply{\zeta} f_E(\z,x,t) - \reply{\gamma} f_I(\z,x,t)
                     +\frac{1}{\tau} Q_I(f_I)(\z,x,t),\\
\pd{f_R(\z,x,t)}{t}  &= \gamma f_I(\z,x,t)
                     + \frac1\tau Q_R(f_R)(\z,x,t),
\end{aligned}
\right.
\label{eq:seisc1}
\end{equation}
where \reply{the operators $Q_J(f_J)$ characterizes the emergence of the distribution of social contacts in the compartment $J \in \mathcal C$.} The transmission of the 
infection is governed by the local incidence rate defined as 
\begin{equation}\label{eq:K}
K(f_S, f_I)(\z,x,t) = f_S(\z,x,t) \intRp \kappa(x,x\*) 
f_I(\z,x\*,t)\, d x\*
\end{equation}
where $\kappa(x,x\*)$ is a nonnegative contact 
function measuring the impact of contact rates among different compartments. A 
leading example for $\kappa(x,x_*)$ is obtained by choosing 
\[
\kappa(x,x\*) = \beta x^\alpha x_*^\alpha, 
\] 
with $\beta>0$ and $\alpha>0$.  In the following, we will stick to the case 
$\alpha=1$ for simplicity so that
\begin{equation}
\label{eq:Kalpha1}
K(f_S,f_I)(\z,x,t) = \beta x f_S(\z,x,t) m_I(\z,t) \rho_{I}(\z,t). 
\end{equation}
This choice formalizes an incidence rate that is proportional on the product of 
the number of contacts of susceptible and infected people. 
The other epidemiological parameters characterizing the spread of the disease 
are $\zeta>0$, the transition rate of exposed individuals to the infected 
class and $\gamma>0$, the recovery rate.  The introduced 
parameters have been summarized in Table \ref{tab:parameters}.
\begin{table}
\begin{center}
\begin{tabular}{cl}
\toprule
 Parameter & Definition\\
\toprule
$\beta$ & contact rate between susceptible and infected individuals\\
$1/\zeta$ & average latency period\\
$1/\gamma$ & average duration of infection\\
\bottomrule
\end{tabular}
\end{center}
\caption{Parameters definition in the SEIR model \eqref{eq:seisc1}.}
\label{tab:parameters}
\end{table}
Finally, the relaxation parameter $0 < \tau \ll 1$ represents the frequency at 
which the agents modify their contact distribution in response to the epidemic 
dynamics. As we will see, we are assuming that the social dynamics is much 
faster than the epidemic dynamics \cite{Zhang20}.

\subsection{Contact formation dynamics}

The total number of contacts can be viewed as a result of the superimposition of repeated updates and possible deviations due to aleatoric uncertainty, see 
\cite{Pareschi13, Furioli17}. In particular, similarly to \cite{Dimarco21, 
Dimarco22} we consider the following microscopic scheme
\begin{equation} \label{eq:binint}
x_J^\prime = x - \Phi_\epsilon^\delta(\z,x/m_J)x + \eta_\epsilon x,    
\end{equation}
where \reply{$x_J'-x$ is the elementary variation of the number of contacts} and $\Phi_\epsilon^\delta$ defines the transition function 
\begin{equation} \label{eq:tr_func}
\Phi_\epsilon^\delta(\z,s) = \mu 
\frac{e^{\epsilon(s^\delta-1)/\delta}-1}{e^{\epsilon(s^\delta-1)/\delta}+1}, 
\qquad s=x/m_J,
\end{equation}
with $\epsilon>0$. In \eqref{eq:tr_func} we introduced a constant $\mu>0$ linked to the maximum variability of the function and the centered random variable $\eta_\epsilon$ such that $\left \langle \eta^2_\epsilon \right\rangle =\epsilon \sigma^2$, being $\left \langle\cdot \right\rangle$ the expectation with respect to the introduced random variable. The constant $\epsilon>0$ tunes the strength of 
interactions. We remark that the microscopic model \eqref{eq:binint} depends on 
a parametric uncertainty and $\delta = \delta(\z)$, such that  $\delta(\z)\in 
[-1,1]$, for any $ \z \in \R^{d_{\z}}$. The transition function 
\eqref{eq:tr_func} is defined such that it is simpler to reach a high number of 
daily contacts while it is very unlikely to go under a certain threshold. This 
type of asymmetry is typical of human and biological phenomena as shown e.g. in 
\cite{Preziosi21, Medaglia22, Gualandi19, Dimarco21, Dimarco22, Toscani21}. In 
the regime $\epsilon \ll 1$ we have
\begin{equation}
\label{eq:Phiscale}
\Phi_\epsilon^\delta(\z,x/m_J) \approx 
\frac{\epsilon\mu}{2\delta(\z)}\biggl[\biggl(\frac{x}{m_J}\biggr)^{\delta(\z)} -1\biggr]
\coloneqq \epsilon\,\Phi^\delta(\z,x/m_J).
\end{equation}
Note also that the function $\Phi_\epsilon^\delta$ is such that
\[
-\mu \le \Phi_\epsilon^\delta(\z,x/m_J) \le \mu
\]
for all $\delta(\z) \in [-1,1]$ and $\epsilon>0$. Clearly, the choice $\mu<1$ 
implies that, in absence of randomness, the value \reply{$x_J^\prime$} remains positive if 
$x$ is positive. It is interesting to observe that $\Phi_\epsilon^\delta$ is 
asymmetric around that value $x/m_J$ with respect to different distributions of 
$\delta$. In particular, $\Phi_\epsilon^\delta$ is increasing and convex for 
any $x/m_J \le 1$ if $\delta>0$ whereas, if $\delta<0$, the transition function 
becomes concave in an interval $[0,\bar x]$, $\bar x/m_J<1$, and then convex. 

Once the microscopic process \eqref{eq:binint} is given, the time evolution of 
the distribution of the number of social contacts $f$ follows by resorting to 
kinetic collision-like approaches, see \cite{cerc,Pareschi13}, that quantify 
the variation of the density of the contact variable in terms of an 
interaction operator, for any time $t\ge0$. The time evolution of $f$ is given by the following kinetic equation written in weak form 
\reply{
\[
\frac{d}{dt} \intRp \phi(x)f_J(\z,x,t)\, dx = \dfrac{1}{\epsilon} \int_{\mathbb R^+}\varphi(x) Q(f_J)(\z,x,t)dx
\]
where
\(
\int_{\mathbb R^+}\varphi(x) Q(f_J)(\z,x,t)dx =  \intRp B(\z,x)
\bracket{\phi(x^\prime_J)- \phi(x)}f_J(\z,x,t)\, dx,
\label{eq:boltzmann}
\)}
where we indicated with \reply{$\phi: \mathbb R^+ \to \mathbb R$, $\varphi(x)\in\mathcal{C}^\infty(\mathbb R^+)$ an observable quantity}.
 In the following, we will consider an uncertain interaction  kernel expressing 
 a multiagent system in which the frequency of changes in the number of social 
 contacts depends on $x$ through the following law 
\begin{equation}
\label{eq:Bker}
B(\z,x) = x^{-\alpha(\delta(\z))}, 
\end{equation}
being in particular  
\[
\alpha(\delta(\z))=\frac{1+\delta(\z)}{2} \geq 0, \qquad \textrm{for any} \quad\delta(\z)\in [-1, 1].
\]
We observe that the kernel \eqref{eq:Bker} mimics the fact that a priori 
information on the frequency of interaction of a system of agents is missing, 
see \cite{medaglia22_PDEA}. 

\begin{remark}\label{rem:1}
If we consider $\varphi(x) = 1$ in \eqref{eq:boltzmann} we easily get the 
conservation of the mass. Furthermore, if $\varphi(x) = x$ we have
\[
\dfrac{d}{dt}m_J(\z,t) =- \dfrac{1}{\epsilon} \int_{\mathbb R^+} 
x^{1-\alpha(\delta)}  \Phi_\epsilon^\delta(\z,x/m_J) f_J(\z,x,t)dx. 
\]
If $\epsilon \ll 1$ from \eqref{eq:Phiscale} we get
\[
\dfrac{d}{dt}m_J(\z,t) = \dfrac{\mu}{2\delta(\z)}\int_{\mathbb R^+} 
x^{1-\alpha(\delta)}\left[ \left(\dfrac{x}{m_J} \right)^{\delta(\z)}-1\right] 
f_J(\z,x,t)dx.
\]
Therefore, if we exploit the form of the interaction kernel \eqref{eq:Bker} we 
have that $m_J(\z,t)$ is a conserved quantity of \eqref{eq:boltzmann} if 
$\delta$ is a discrete random variable such that  $\delta(\z)\in\{-1,1\}$ for 
all $\z \in \mathbb R^{d_\z}$. A possible example that we will study in the following is given by $\delta(\z ) = 1-2\z$, where 
$\z \sim \textrm{Bernoulli}(p)$. 
\end{remark}
\subsection{Fokker-Planck scaling and steady states}
In general, it is difficult to compute analytically the equilibrium state of the 
 kinetic model \eqref{eq:boltzmann}. A possible approach has its 
roots in the so-called grazing collision limit of the classical Boltzmann 
equation \cite{cerc,villaniARMA}. In this direction, a deeper insight on the 
steady states can be obtained through a quasi-invariant technique 
\cite{Furioli17,Pareschi13,Tos}. The goal is to derive  a simplified 
Fokker-Planck model from the introduced Boltzmann-type dynamics. For such 
surrogate model, the study of asymptotic properties is much easier.
The idea is to scale simultaneously interactions and interaction frequency. 
Hence, the equilibrium in contact distribution is reached faster than the time scale of the epidemic dynamics. In details, assuming \reply{$\varphi \in \mathcal{C}^\infty$} we may observe that for $\epsilon \ll 1$ the difference $x^\prime_J - x$, $J \in 
\C$,  is small and we can perform a Taylor expansion 
\[
\phi(x^\prime_J) - \phi(x) = (x^\prime_J - x) \dfrac{d}{dx}\phi(x) + 
\frac12(x^\prime_J - x)^2 \dfrac{d^2}{dx^2}\phi(x) + \frac13(x^\prime_J - x)^3 
\dfrac{d^3}{dx^3}\phi(\hat x), 
\]
where $\hat x \in (\min \{x,x^\prime_J\}, \max\{x,x^\prime_J\})$. 
Plugging the above expansion in the interaction operator $Q_J(f_J)(\z,x,t)$ in 
\eqref{eq:boltzmann} and thanks to the scaling \eqref{eq:Phiscale} we get
\begin{multline} \label{eq:boltzmann_taylor}
\frac{d}{dt} \intRp \phi(x)f_J(\z,x,t)\, dx = \intRp \Phi^\delta(\z,x/m_J) 
x^{1- \alpha(\delta)} f(\z,x,t) \dfrac{d}{dx}\phi(x)\, dx\\ + \frac{\sigma^2}2 
\intRp x^{2- \alpha(\delta)}\dfrac{d^2}{dx^2} \phi(x)\, dx + R_\phi(f)(\z,x,t),
\end{multline}
 where we have defined the remainder term
\begin{multline}
R_\phi(f)(\z,x,t) = \intRp \epsilon \lrp{ \Phi^\delta(\z,x/m_J) }^2 x^{1- 
\alpha(\delta)} f(\z,x,t) \phi''(x)\, dx\\ + \dfrac{1}{\epsilon}\intRp 
\left\langle -\epsilon\,\Phi^\delta(\z,x/m_J) + \eta_\epsilon x \right\rangle^3 
x^{1- \alpha(\delta)} f(\z,x,t) \phi'''(\hat{x})\, dx    .
\end{multline}
Assuming $\left\langle|\eta_\epsilon^3|\right\rangle<+\infty$ we can prove 
that, in the limit^^>$\epsilon \to 0^+$, the remainder vanishes thanks to the 
smoothness of the function^^>$\phi$ proceeding as in \cite{cordier_05}. Hence, 
in the quasi-invariant scaling regime, we can show that the solution to model 
\eqref{eq:boltzmann_taylor} converges to 
\begin{equation}\label{eq:boltzmann_limit}
\begin{split}
\dfrac{d}{dt} \intRp \varphi(x)f_J(\z,x,t)dx =& \intRp \dfrac{\mu}{2\delta} 
x^{1-\alpha(\delta)}\lrp*{ \lrp*{\frac{x}{m_J}}^\delta-1 } 
\dfrac{d}{dx}\varphi(x) f_J(\z,x,t) dx \\
&+\dfrac{\sigma^2}{2} \intRp x^{2-\alpha(\delta)} \dfrac{d^2}{dx^2}\varphi(x) 
f_J(\z,x,t)dx
\end{split} 
\end{equation}
Integrating back by parts \eqref{eq:boltzmann_limit} we obtain the 
Fokker-Planck model 
\begin{equation}
\begin{split}
\partial_t f_J(\z,x,t) &= \reply{ \bar Q(f_J)(\z,x,t)} \\
&=  \frac{\mu}{2\delta}\partial_x \left[ 
x^{1-\alpha(\delta)}\lrp*{ \lrp*{\frac{x}{m_J}}^\delta-1 }f_J(\z,x,t)\right] + 
\frac{\sigma^2}{2}\partial^2_{x}\lrp*{x^{2-\alpha(\delta)}f_J(\z,x,t) } 
\end{split}
\label{eq:FP1}
\end{equation}
complemented by no-flux boundary conditions
\(
\begin{split}
\left.\frac{\mu}{2\delta} x^{1-\alpha(\delta)}\biggl[ 
\biggl(\frac{x}{m_J}\biggr)^\delta-1 \biggr]f_J(\z,x,t) + 
\frac{\sigma^2}{2}\partial_x (x^{2-\alpha(\delta)}f_J(\z,x,t))\right|_{x=0} = 0 
\\
x^{2-\alpha(\delta)}f_J(\z,x,t) \Bigg|_{x=0}=0.
\label{eq:boundcond}
\end{split}\)

We can observe now that the steady state of equation^^>\eqref{eq:FP1} depends 
on the parametric uncertainty of the model and is given by 
\( \label{eq:finf}
f^{\infty}_J(\z,x) = C_{\delta,\sigma^2,\mu,m_J}(\z) x^{\frac{\mu}{\sigma^2 
\delta(\z)} - 2 + \alpha(\delta(\z))} \exp\left\{ 
-\dfrac{\mu}{\sigma^2\delta(\z)^2}\lrp*{\dfrac{x}{m_J} }^{\delta}\right\},
\)
corresponding to generalized Gamma density with $C_{\delta,\sigma^2,\mu,m_J}>0$ normalization constant. In particular, we can observe that in the limit 
$\delta \to 0$ we get
\(
\label{eq:finf_delta0}
\finf[0] (x) = C^{(0)}_{\sigma^2,\mu,m_J} x^{3/2} \exp\left\{ 
-\dfrac{\mu}{2\sigma^2} \log^2\lrp*{\frac{x}{m_J}} \right\},
\)
 where again $C^{(0)}_{\sigma^2,\mu,m_J}>0$ is a normalization constant. Whereas, if $\delta(\z)\equiv1$ from \eqref{eq:finf} we get
\begin{equation}
f_{J,1}^\infty(x) = \frac{\lambda^\lambda}{(m_J)^\lambda \Gamma(\lambda)} 
\frac1{x^{1-\lambda}} \exp\set[\bigg]{-\frac{\lambda x}{ m_J}}, \qquad \lambda 
= \mu/\sigma^2,
\label{eq:finf_deltap}
\end{equation}
which is a Gamma distribution. On the other hand, if $\delta(\z) \equiv -1$ 
from \eqref{eq:finf} we get
\begin{equation}
f_{J,-1}^\infty(x) = \frac{(\lambda m_J)^{\lambda+1}}{\Gamma(\lambda +1)} 
\frac1{x^{2+\lambda}} \exp\set[\bigg]{-\frac{\lambda m_J}x},\qquad \lambda = 
\mu/\sigma^2, 
\label{eq:finf_deltam}
\end{equation}
corresponding to an inverse Gamma distribution.

More generally, we may observe that the distribution \eqref{eq:finf} exhibits 
different behaviors depending on the uncertain parameter $\delta(\z)$. In 
particular, for each realization of the random variable $\delta(\z)$ such that  
$\delta<0$ the equilibrium density exhibits fat tails with a polynomial 
decrease for $x\to+\infty$. On the other hand,  for each realization of the 
random variable $\delta(\z)$ such that  $\delta\ge0$, the equilibrium density is 
characterized by slim tails. From the modelling point of view, a  fat-tailed 
distribution of contacts defines a society where a non-negligible portion of 
agents has a high number of contacts. Therefore, the fact that the parameter 
$\delta$ characterizing the tails of the distributions is uncertain means that 
we take into account the lack of knowledge on the behaviour of the society.

\reply{
\begin{remark}
In the present context we have neglected effects related to opinion-type dynamics that may influence the process of contact formation. Recent experimental results have shown that social norm changes are often triggered by opinion alignment phenomena. In particular, the perceived adherence of individuals’ social network has a strong impact on the effective support of protective behaviour. Therefore, the individual responses to threat are a core question to set up effective measures in the presence of cases escalation. 
\end{remark}}

\subsection{Uniqueness of the solution} 
\label{sec:uniqueness}
In this subsection, we prove some properties of the solutions of the Cauchy 
problem^^>\eqref{eq:seisc1} for any $\z \in \R^{d_{\z}}$. 
Let us first concentrate on the Cauchy problem defined by the 
Fokker-Planck-type problem \eqref{eq:FP1} with given initial condition 
$f_J(\z,x,0)\ge 0$. 
We may apply the arguments of \cite{Escobedo91, Carrillo08} to show the 
positivity of the solution of \eqref{eq:FP1}.

\begin{proposition}\label{thm:unique}
Let $f_J$ be a solution of the Cauchy problem 
\(\label{eq:CP1} 
\partial_t f_J(\z,x,t) = \reply{\bar{Q}(f_J)(\z,x,t)}, \quad J\in \set{S,E,I,R},
\)
{\replydec
where 
\[
\begin{aligned}
    \bar{Q}(f_J)(\z,x) &= \partial_x \left[A_J(\z,x)f_J(\z,x,t) + \partial_x^2(B_J(\z,x)f(\z,x,t))\right],
\end{aligned}
\]
and
\[
A_J(\z,x)    = 
\frac{\mu}{2\delta}x^{1-\alpha(\delta)}\biggl[\biggl(\frac{x}{m_J}\biggl)^\delta - 1\biggr],\quad
    B_J(\z,x)    = \frac{\sigma^2}{2} x^{2-\alpha(\delta)},
\]}
with initial condition  $f_J^0=f_J(\z,x,0)$. If $f_J^0 \in L^1(\Rp)$ for all 
$\z \in \mathbb R^{d_\z}$ then $\int_{\mathbb R^+}|f_J|dx$ is non-increasing for 
all $\z \in \mathbb R^{d_\z}$ and $t \ge 0$. 
\end{proposition}
\begin{proof}
Let us consider a positive constant $\epsilon>0$. We introduce an increasing 
approximation of the $\sign$ function 
$\signep(f_J)(\z,x)$, $\z \in \mathbb R^{d_\z}$, $x \in \Rp$, with $J \in 
\set{S, E, I, R}$, and define the
approximation $\abs{f_J}_\epsilon(\z,x)$ of $\abs{f_J}(\z,x)$ by the primitive 
of 
$\signep(f_J)(\z,x)$. Hence, we write the Fokker-Planck equation in weak form 
where we consider the smooth function $\phi = \signep(f_J)(\z,x)$ to obtain
\(
\begin{split}
\frac{d}{dt} \intRp \abs{f_J}_\epsilon(\z,x) \, dx &=
\intRp \signep (f_J)(\z,x) \partial_x [A_J(\z,x)f_J(\z,x)]\, dx \\
&\quad+\intRp \signep (f_J)(\z,x) \partial^2_x [B_J(\z,x)f_J(\z,x)]\, dx\\
&= -\intRp [\signep'(f_J) (\z,x) \partial_x f_J(\z,x)]A_J(\z,x)f_J(\z,x)\, dx\\
&\phantom{{}=} -\intRp [\signep'(f_J) (\z,x) \partial_x 
f_J(\z,x)]\partial_x[B_J(\z,x)f_J(\z,x)]\, dx,
\end{split}
\)
where we recall that $\delta = \delta(\z)$. Since the boundary terms $\signep 
(f_J)(\z,x) A_J(\z,x)f_J(\z,x)\mid_{x= 0}^{+ \infty} $  
and $\signep (f_J)(\z,x) \partial_x [B_J(\z,x)f_J(\z,x)]\mid_{x= 0}^{+ \infty} 
$ vanish in view of the boundary conditions, we have
\(
\begin{split}
\frac{d}{dt} \intRp \abs{f_J}_\epsilon(\z,x) \, dx = &-
\intRp \sed (f_J)(\z,x)f_J(\z,x) \partial_x f_J(\z,x) [A_J(\z,x) + \partial_x 
B_J(\z,x)]\, 
dx \\
&-\intRp \sed(f_J) (\z,x) [\partial_x f_J(\z,x)]^2 B_J(\z,x)\, dx.
\end{split}
\)
Next we observe that for all $\z \in \mathbb R^{d_\z}$
\(
\sed(f_J)(\z,x) f_J(\z,x)\partial_x f_J(\z,x) = \partial_x 
[f_J(\z,x)\signep(f_J)(\z,x) - 
\abs{f_J}_\epsilon(\z,x)].
\)
Therefore we have
\(
\begin{split}
\frac{d}{dt} \intRp \abs{f_J}_\epsilon(\z,x) \, dx = &-
\intRp \partial_x[f_J(\z,x) \signep(f_J)(\z,x) - \abs{f_J}_\epsilon(\z,x)]
(A_J(\z,x) + \partial_x B_J(\z,x))\, dx\\
&- \intRp \sed(f_J) (\z,x) [\partial_x f_J(\z,x)]^2 B_J(\z,x)\, dx.
\end{split}
\)
Hence, integrating by parts the first term of the above equation we obtain 
that in the limit $\epsilon\to 0^+$ such term vanishes and
\(
\frac{d}{dt} \norma{f_J(\z,x)}_{L^1} \le 0,
\)
for all $\z \in \mathbb R^{d_\z}$ and for all $J \in \I$. Therefore, for all $t 
\ge 0$, if we take another solution  $g_J(x,t)$ of the Cauchy 
problem^^>\eqref{eq:seisc1} with initial condition $g^0_J=g_J(\z,x,0)$, we have
\(
\norma{f_J(\z,x,t) - g_J(\z,x,t)}_{L^1} \le \norma{f_J(\z,x,0) - 
g_J(\z,x,0)}_{L^1}.\qedhere
\)
\end{proof}
\begin{cor}\label{cor:FPpos}
Let $f_J$ be a solution of the Cauchy problem^^>\eqref{eq:CP1} with initial 
condition $f_J(\z,x,0) \in L^1(\Rp)$. If $f_J(\z,x,0) \ge 0 $ for any $\z \in 
\mathbb R^{d_\z}$ and $x\in\Rp$ a.e., then 
$f_J(\z,x,t) \ge 0$ a.e., for all $t \ge 0$ and $\z \in \mathbb R^{d_\z}$.
\end{cor}
\begin{proof}
The result follows from a similar proof presented in \cite{Carrillo08}. 
\end{proof}
Now, we concentrate on the epidemic dynamics proving the positivity of the 
solution of the SEIR-type compartmental system in absence of the collision 
operators $Q_J$, $J\in \C$, see \cite{El11}. 

\begin{proposition}\label{lem:seirpos}
Let $f_J(\z,x,t)$, $x \in \mathbb R^+$, $\z \in \mathbb R^{d_\z}$, $J \in \C$ be 
a solution of the Cauchy problem
\begin{equation}
\left\lbrace
\begin{aligned}
\pd{f_S(\z,x,t)}{t}  &= -K(f_S, f_I)(\z,x,t) , \\ 
\pd{f_E(\z,x,t)}{t}  &= K(f_S, f_I)(\z,x,t) - \zeta(x) f_E(\z,x,t),\\
\pd{f_I(\z,x,t)}{t}  &= \zeta(x) f_E(\z,x,t) - \gamma(x)f_I(\z,x,t),\\
\pd{f_R(\z,x,t)}{t}  &= \gamma(x) f_I(\z,x,t),
\end{aligned}
\right.
\label{eq:sys_pos}
\end{equation} 
with the initial data^^>$f_J(\z,x,0)\geq 0$ for all $x\geq 0$ and $\z \in 
\mathbb R^{d_\z}$, and $K(f_S,f_I)$ defined as 
\[
K(f_S,f_I)(\z,x,t)  =  f_S(\z,x,t) \int_{\mathbb R^+} 
\kappa(x,x_*)f_I(\z,x_*,t)dx_*,
\] 
with $\kappa\ge 0$ for all $x,x_*\in \mathbb R^+\times \mathbb R^+$. Then 
$f_J(\z,x,t) \ge 0$ for all $\z \in \mathbb R^{d_\z}$, $x\in \mathbb{R}^+$ and 
$t \ge 0$.
\end{proposition}
\begin{proof}
We proceed by contradiction. Let us suppose that there exists a time instant 
$t_0>0$ such that there exists a point $x_0>0$ such that
\[
f_S(\z,x_0,t_0)=0, \quad \partial_t f_S(\z,x_0,t_0) < 0, \quad f_S(\z,x,t)\geq 
0 \quad \textrm{for all}\; t\in[0,t_0), 
\]
and for all $x \in \mathbb R^+$, $\z \in \mathbb R^{d_\z}$.
Then, $f_E(\z,x,t)\geq 0$ for all $t\in[0,t_0)$ and $x\geq 0$. If not, there 
must be a time $t_1\in[0,t_0)$ such that there exists a value $x_1>0$ for which 
\[
f_E(\z,x_1,t_1)=0, \quad \partial_t f_E(\z,x_1,t_1) < 0, \quad f_E(\z,x,t)\geq 
0 \quad \textrm{for all}\; t\in[0,t_1),
\]
and for all $x \in \mathbb R^+$, $\z \in \mathbb R^{d_\z}$. Hence, integrating 
the third equation of \eqref{eq:sys_pos} we get
\[
f_I(\z,x,t)=f_I(\z,x,0)\,e^{-\gamma t} + \zeta \int_0^t f_E(\z,x,s)\, 
e^{-\gamma(t-s)} ds \geq 0 \quad \textrm{for all}\; t\in[0,t_1).
\]
Then we have
\[
\partial_t f_E(\z,x_1,t_1)=\beta x_1 f_S(\z,x_1,t_1) \intRp y f_I(\z,y,t_1) dy 
\geq 0
\]
that is not coherent with the hypothesis. As a consequence, it holds 
$f_E(\z,x,t)\geq 0$ for all $t\in[0,t_0)$, all $x\geq0$ and all 
$\z\in\mathbb{R}^{d_z}$. Furthermore, we also have that $f_I(\z,x,t)\geq 0$ for 
all $t\in[0,t_0)$ and $x\geq0$. If not, there should be a time $t_2\in[0,t_0)$ 
such that there exists a position $x_2>0$ for which 
\[
f_I(\z,x_2,t_2)=0, \quad \partial_t f_I(\z,x_2,t_2) < 0, \quad f_I(x,t)\geq 0 
\quad \textrm{for all}\; t\in[0,t_2),
\]
for all $x \ge 0$ and $\z \in \mathbb R^{d_\z}$. Proceeding as before we get
\[
\partial_t f_I(\z,x_2,t_2) = \zeta f_E(\z,x_2, t_2) \geq 0
\]
that is not coherent with the hypothesis. It follows that $f_I(\z,x,t)\geq 0$ 
for all $t\in[0,t_0)$ and $x\geq0$.
In view of the results on $f_E$ and $f_I$, we get $f_R(x,t)\geq 0$ for all 
$t\in[0,t_0)$ and $x\geq0$.

To conclude, we observe that
\[
\partial_t f_S(\z,x_0,t_0) =  \gamma f_I(\z,x_0, t_0) \geq 0 
\]
which is the desired contradiction. Therefore, $f_S(\z,x,t)\geq 0$ for all 
$t\in[0,t_0)$ and $x\geq0$.
\end{proof}

Once proved the positivity of the contact formation model and of the 
epidemiological dynamics, we can conclude that the solution of the general 
Cauchy problem^^>\eqref{eq:seisc1} with a non-negative initial data $f_J$ is 
positive a.e.\ for all $t\ge 0$ and $\z \in \mathbb R^{d_\z}$. 

In the following we concentrate on the uniqueness of the solution of the 
introduced model. 
\begin{theorem}[Uniqueness of the solution]
Let $f_J,\,g_J$, with $J\in\{S,E,I,R\}$, be two solutions of the Cauchy 
problem
{\replydec
\[
\left\lbrace
\begin{aligned}
\pd{f_S(\z,x,t)}{t}  &= -K(f_S, f_I)(\z,x,t) 
                     + \frac{1}{\tau} Q_S(f_S)(\z,x,t),\\ 
\pd{f_E(\z,x,t)}{t}  &= K(f_S, f_I)(\z,x,t) - \zeta f_E(\z,x,t)
                     + \frac{1}{\tau} Q_E(f_E)(\z,x,t),\\
\pd{f_I(\z,x,t)}{t}  &= \zeta f_E(\z,x,t) - \gamma f_I(\z,x,t)
                     +\frac{1}{\tau} Q_I(f_I)(\z,x,t),\\
\pd{f_R(\z,x,t)}{t}  &= \gamma f_I(\z,x,t)
                     + \frac1\tau Q_R(f_R)(\z,x,t),
\end{aligned}
\right.
\]
where we take $K(f_S, f_I)$ as in Proposition^^>\ref{lem:seirpos} and
}
constant, 
positive epidemiological parameters $\beta$, $\zeta$, $\gamma>0$. Furthermore, we 
assume \reply{the existence of a positive constant $\bar{\kappa}>0$ such that} $\|\kappa(x,x_*)\|_{L^{\infty}}\le \bar{\kappa}$.  If $f_J(\z,x,0)\in 
L^1$ and $g_J(\z,x,0)\in L^1$, then there exists $C^{\max}>0$ such that 
\[
\sum_{J \in \C} \norma{f_J(\z,x,t) - g_J(\z,x,t)}_{L^1(\mathbb R^+)} \le 
e^{C^{\max}t} \sum_{J \in \C}\norma{f_J(\z,x,0) - g_J(\z,x,0)}_{L^1(\mathbb 
R^+)}.
\]
\end{theorem}
\begin{proof}
In the following, we drop the dependence on $x\in \mathbb R^+$, $t\ge 0$ and 
$\z \in \mathbb R^{d_\z}$ for brevity.
We first observe that the difference between two solutions is itself solution 
of the system
\[
\left\lbrace
\begin{aligned}
\pd{(f_S - g_S)}{t}  &= -[K(f_S, f_I) - K(g_S, g_I)] 
                     + \frac{1}{\tau} Q(f_S - g_S),\\ 
\pd{(f_E - g_E)}{t}  &= [K(f_S, f_I) - K(g_S, g_I)] - \zeta[f_E - g_E]
                     + \frac{1}{\tau} Q(f_E - g_E),\\
\pd{(f_I - g_I)}{t}  &= \zeta[f_E - g_E] - \gamma[f_I - g_I]
                     +\frac{1}{\tau} Q(f_I - g_I),\\
\pd{(f_R - g_R)}{t}  &= \gamma[f_I - g_I]
                     + \frac1\tau Q(f_R - g_R).
\end{aligned}
\right.
\]
From the proof of \reply{Proposition}^^>\ref{thm:unique}  we get
\[
\left\lbrace
\begin{aligned}
\frac{d}{dt}\intRp\abs{f_S - g_S}\, dx  &\le -\intRp\abs{K(f_S, f_I) - K(g_S, 
g_I)}\, dx,\\ 
\frac{d}{dt}\intRp\abs{f_E - g_E}\, dx  &\le \intRp\abs{K(f_S, f_I) - K(g_S, 
g_I)}\, dx
                          + \zeta\norma{f_E - g_E}_{L^1(\mathbb R^+)},\\
\frac{d}{dt}\intRp\abs{f_I - g_I}\, dx  &\le \zeta\norma{f_E - 
g_E}_{L^1(\mathbb R^+)}
                          + \gamma\norma{f_I - g_I}_{L^1(\mathbb R^+)},\\
\frac{d}{dt}\intRp\abs{f_R - g_R}\, dx  &\le \gamma\norma{f_I - 
g_I}_{L^1(\mathbb R^+)}.\\
\end{aligned}
\right.
\]
Now, we can rewrite $K(f_S, f_I) - K(g_S, g_I)$ as follows
\[
K(f_S, f_I) - K(g_S, g_I) = (f_S - g_S) \intRp \kappa(x,x_*)f_I(x_*)\, dx_*
                            + g_S \intRp \kappa(x,x_*) (f_I - g_I)(x_*) \, dx_*,
\]
from which we have
\[
\begin{split}
\intRp \abs{K(f_S, f_I) - K(g_S, g_I)}\, dx
& \le \int_{\mathbb R^+} \left|(f_S-g_S) \int \bar \kappa f_Idx_* + g_S 
\int_{\mathbb R^+}\bar \kappa (f_I-g_I)dx_*\right| dx \\
&\le c\left( \| f_S-g_S\|_{L^1(\mathbb R^+)} + \| f_I-g_I\|_{L^1(\mathbb R^+)} 
, \right)
\end{split}
\]
with $c>0$. This allows us to write
\(
\left\lbrace
\begin{aligned}
\frac{d}{dt}\norma{f_S - g_S}_{L^1}(t)  &\le c\left( \| f_S-g_S\|_{L^1(\mathbb 
R^+)} + \| f_I-g_I\|_{L^1(\mathbb R^+)}\right) , \\ 
\frac{d}{dt}\norma{f_E - g_E}_{L^1}(t)  &\le c\left( \| f_S-g_S\|_{L^1(\mathbb 
R^+)} + \| f_I-g_I\|_{L^1(\mathbb R^+)}\right) 
                                        + \zeta\norma{f_E - g_E}_{L^1},\\
\frac{d}{dt}\norma{f_I - g_I}_{L^1}(t)  &\le \zeta\norma{f_E - g_E}_{L^1}
                                        + \gamma\norma{f_I - g_I}_{L^1},\\
\frac{d}{dt}\norma{f_R - g_R}_{L^1}(t)  &\le \gamma\norma{f_I - g_I}_{L^1}.
\end{aligned}
\right.
\)
Then there exists $C^{\max}>0$ such that 
\(
\frac{d}{dt}\sum_{J\in \C} \norma{f_J - g_J}_{L^1(\mathbb R^+)} \le C^{\max}
                           \sum_{J \in \C} \norma{f_J - g_J}_{L^1(\mathbb R^+)},
\)
which, by Gronwall's inequality, gives the claim.
\end{proof}

\section{Selective control of the kinetic epidemic model}\label{sect:3}
In Section \ref{sect:2} we introduced and discussed a variety of kinetic 
models to describe the contact formation dynamics in a society. The main brick 
of the construction relies on the choice of the transition functions 
\eqref{eq:tr_func} embedding uncertainties in the elementary updates 
\eqref{eq:binint}, and characterizing the growth in terms of an uncertain 
parameter $\delta = \delta(\z)$. In particular, it was shown that, for negative 
values of the parameter $\delta$, the resulting equilibrium contact 
distribution is given by a distribution with polynomial tails  
\eqref{eq:finf_deltam}. On the other hand, slim tailed distributions can be 
obtained for positive values of $\delta$, see \eqref{eq:finf_delta0}-\eqref{eq:finf_deltap}.

In this section, we will investigate the possibility to control the dynamics of 
contact formation mimicking the action of non-pharmaceutical interventions 
which should then mitigate the risk factors linked to the transmission of the 
infection. The new kinetic description allows to enlighten the effects of 
interventions of the policy maker by acting on the contact distribution of the 
society of which partial information is available. It is worth to mention that the control of multiagent systems has been recently 
investigated as a natural follow-up issue in the description and modeling of 
their self-organization ability, see e.g. 
\cite{ACFK,albi18,Albi14,FPR14,Medaglia22} and the references therein.


\subsection{The controlled model}\label{subsect:controlled}
To mimic the action of non-pharmaceutical interventions, we add to the 
microscopic evolution of the social contacts a second update dynamics, 
implementing an additive control term^^>$u$, to limit selectively the social 
activities, see  \cite{albi18,Preziosi21}. Hence, the contact formation is 
influenced by the uncertain dynamics defined in \eqref{eq:binint} and, in 
parallel, by the elementary interaction under control
\begin{equation} \label{eq:interaction_controlled}
x''_J= x + \sqrt{\epsilon\tau} S(x) u, 
\end{equation}
where \reply{$x_J''-x$ is variation of social contacts in the presence of the control $u$ and} $S(x)\geq0$ is a selective function which depends on the number of 
contacts. 

The small parameters $\tau$ and $\epsilon$ represent, respectively, 
the speed at which the contact dynamics equilibrium is reached and the limit from the Boltzmann dynamics to the Fokker-Planck one \cite{Dimarco22}. \reply{Two different speed values need to be considered since it is reasonable to assume that such interventions share the time scale with the epidemics, wich is much faster than the contact formation process.} We 
remark also that this second interaction scheme is independent by the uncertain parameter $\z\in \mathbb R^{d_\z}$ and it is linked to a new additive Boltzmann collisional operator which scales with the epidemic dynamics.

The optimal control $u^*$ is such that
\begin{equation}
u^* = \arg \min_{u\in\mathcal{U}} \mathcal{J}_J(x^{\prime\prime}_J,u),
\label{eq:mincontrol}
\end{equation}
under the constraint \eqref{eq:interaction_controlled}, where 
$\mathcal{U}$ is the set of the admissible controls, \reply{i.e., the set of controls such that $x_J''\geq 0$}. We define the cost 
$\mathcal{J}_J$ as follows
\(
\mathcal{J}_J(u,x'') =  (x''_J - x_{T,J})^2 + \kappa |u|^p, \qquad 
J\in\{S,E,I,R\},
\label{eq:Jfunctional}
\)
being $\kappa>0$ a penalization coefficient and $x_{T,J}>0$ the desired target 
number of contacts to reach in each compartment. We remark that the introduced penalization can depend by the compartment of the agent and that the control 
obtained from \eqref{eq:mincontrol} subject to 
\eqref{eq:interaction_controlled} is independent on $\z \in \mathbb R^{d_\z}$. 
Typical choices for the cost function $\mathcal{J}_J$ are obtained for $p=1,2$ 
and a clear analytical understanding is generally difficult for general convex 
functions and suitable numerical method should be developed. 

Let us consider the simple case $p = 2$. Hence, the minimization 
of^^>\eqref{eq:mincontrol} can be achieved via a Lagrangian multiplier 
approach. We define the Lagrangian
\[
\mathcal L (u,x''_J) = \mathcal{J}_J + \theta\left(x''_J - x - 
\sqrt{\epsilon\tau} S(x) 
u \right),
\]
where $\theta \in \R$ is the multiplier associated to the 
constraint^^>\eqref{eq:interaction_controlled}. Then we compute
\[
\left\lbrace
\begin{aligned}
\frac{\partial\mathcal L (u,x''_J)}{\partial u} = 2\kappa u - \theta 
\sqrt{\epsilon\tau} S(x) &= 0\\
\frac{\partial\mathcal L (u,x''_J)}{\partial x''} = 2(x''_J - x_{T,J}) + \theta 
&=0 
\end{aligned}
\right.
\]
which yields the optimal control 
\(
u^* = - \frac{\sqrt{\epsilon\tau}S(x)}{\kappa+\epsilon\tau S^2(x)} ( x - 
x_{T,J}).
\label{eq:ustar}
\)
Thus, plugging $u^*$ into^^>\eqref{eq:mincontrol} we obtain the controlled 
update
\[
x''_J = x - \frac{\epsilon\tau S^2(x)}{\kappa+\epsilon\tau S^2(x)} ( x - 
x_{T,J}),
\]
\reply{which is a non negative quantity, as required.}
The kinetic equation expressing the introduced control strategy \reply{in the presence of elementary transitions of the contact formation dynamics} is a sum of 
collision operators
\(
\label{eq:Bol_JC}
\begin{split}
\dfrac{d}{dt} \int_{\mathbb R^+}\varphi(x) f_J(\z,x,t)dx =& \frac{1}{\tau} 
\int_{\mathbb R^+} B(\z,x) \left\langle \varphi(x_J^\prime) - \varphi(x) 
\right\rangle f_J(\z,x,t) dx  \\
&+ \int_{\mathbb R^+} \bar{B}(\z,x) ( \varphi(x_J^{\prime\prime})- \varphi(x) ) 
f_J(\z,x,t) dx
\end{split}
\)
where the first term on the rhs has been defined in \eqref{eq:boltzmann} and 
the second operator describes the impact of non-pharmaceutical interventions on 
the formation of social contacts. In \eqref{eq:Bol_JC} we have introduced also 
a second kernel $\bar{B}(\z,x)$, in principle different from $B(\z,x)$, 
describing the frequency of interactions of the agents under the action of the 
control.

Similarly to what we have done in the uncontrolled scenario, under the grazing 
limit $\epsilon\to0$ and scaling the penalization as $\kappa=\tau\nu$, $\nu>0$, 
we get a surrogate Fokker-Planck model accounting for an additional drift term 
quantifying the impact of the control
\begin{equation}
\label{eq:surr_FPC}
\partial_t f_J(\z,x,t) =\dfrac{1}{\tau} \partial_x \left[ 
\frac{\mu}{2\delta}x^{1-\alpha(\delta)}\lrp*{ \lrp*{\frac{x}{m_J}}^\delta-1 
}f_J(\z,x,t) + 
\frac{\sigma^2}{2}\partial_{x}\lrp*{x^{2-\alpha(\delta)}f_J(\z,x,t) } \right] + 
C_J(\z,x,t)
\end{equation}
where
\begin{equation}
\label{eq:CJ}
C_J(f_J)(\z,x,t) = \frac{1}{\nu} \partial_x \lrp{ \bar{B}(\z,x) S^2(x) 
(x-x_{T,J}) f_J(\z,x,t) },
\end{equation}
see \cite{Dimarco22}, whose steady state  is given by
\[
\begin{split}
f^{\infty}_{J}(\z,x)  = C_{\delta,\sigma^2,\mu,m_J,\nu} &x^{\frac{\mu}{\sigma^2 
\delta} - 2 + \alpha(\delta)} \exp\left\{ 
-\dfrac{\mu}{\sigma^2\delta^2}\lrp*{\dfrac{x}{m_J} }^{\delta}\right\} \\
&\times\exp\left\{ - \frac{2}{\sigma^2 \nu} \int \bar{B}(\z,x) 
x^{\alpha(\delta)-2} S^2(x) (x-x_{T,J})\, dx \right\}  ,
\end{split}\]
corresponding to a generalized Gamma density. 

\begin{remark}
It is interesting to observe that if $\bar B\equiv 1$ we can easily determine a  
$S(x)$ to force a slim tailed equilibrium even in the case $\delta<0$ for any $\z \in\mathbb R^{d_z}$. In particular, we have that any selection function $S(x)$ with superlogarithmic growth is sufficient to ensure that 
$f_J^\infty(\z, x)$ is slim-tailed.
\end{remark}

\subsection{Damping effects on the model uncertainties} \label{sec:damping}
It is of interest to quantify the effects of the introduced controls 
on the uncertainties of the kinetic model. Under suitable hypothesis, it has 
been observed how the lack of information of system of agents can be dampened 
for small penalizations, see e.g. \cite{Medaglia22,TZ_MCRF}. 
In the following, we concentrate on the damping effects of the control in terms 
of the introduced uncertainties by choosing a Maxwellian kernel for the control 
operator, i.e. $\bar{B}(\z,x)\equiv1$, and considering two possible selective 
functions. We consider the uniform control case $S(x)\equiv1$ and the possible selective control that is increasing with $x \in\mathbb R^+$, $S(x)=\sqrt{x}$.

Let us consider the model 
\eqref{eq:Bol_JC} and we  introduce the time scale $\xi = \epsilon t$. We 
restrict our analysis to the case in which $\delta(\z)$ is a discrete random 
variable such that $\delta(\z) \in \{- 1,1\}$. We recall that the mean is 
conserved in time as observed in Remark \ref{rem:1}. 

By indicating $m_J(\z,\xi) = m_J(\z,t/\epsilon)$ we get
\(
\frac{d}{d\xi} m_J(\z,\xi) = \frac{1}{\epsilon\tau} \intRp B(\z,x)
\bracket{x'_J-x}f_J(\z,x,\xi)\, dx 
+ \frac{1}{\epsilon} \intRp  (x''_J- x)f_J(\z,x,t)\, dx.
\)
Hence, by considering the scaled penalization $\kappa = \nu\tau$ we get in the 
limit $\epsilon\to 0$ 
\[
\frac{d}{d\xi} m_J(\z,\xi) = - \dfrac{1}{\tau}\intRp \!\Phi^\delta(\z,x/m_{J}) 
x^{1-\alpha(\delta)} f_J(\z,x,\xi)dx - \frac1\nu \intRp 
S^2(x) (x-x_{T,J}) f_J(\z,x,\xi)\, dx,
\]
whose large time behavior is
\[
\intRp \Phi^\delta(\z,x/m_{J}) x^{1-\alpha(\delta)} f^{\infty}_{J} (\z,x)\, dx 
= - 
\frac\tau\nu \intRp  S^2(x) (x-x_{T,J}) f^{\infty}_{J} (\z,x) \, dx.
\]
We have
\begin{equation}
\label{eq:minfty_abs}
\begin{split}
\abs*{\intRp \Phi^\delta(\z,x/m_{J}) x^{1-\alpha(\delta)} f^{\infty}_{J} 
(\z,x)\, dx} 
&\leq \intRp \abs{\Phi^\delta(\z,x/m_{J})} x^{1-\alpha(\delta)} f^{\infty}_{J} 
(\z,x)\, 
dx\\
&\leq \mu  \; m^\infty_{1-\alpha(\delta), J}(\z), 
\end{split}
\end{equation}
since $\abs{\Phi^\delta} \le \mu$. \reply{In \eqref{eq:minfty_abs} we used the notation $m_{r,J}^\infty(\z)$ to indicate the moment of order $r>0$ of compartment $J \in \I$ at the equilibrium, i.e.
\[
m_{r,J}^\infty = \int_{\mathbb R^+} x^r f^\infty_J(\z,x)dx.
\]}
Let us consider two cases:
\begin{itemize}
\item If we consider $S(x)\equiv1$, then we get
\[
\intRp (x - x_{T,J}) f^{\infty}_{J} (\z,x) \, dx = m^\infty_J - x_{T,J},
\]
which leads to the estimate
\begin{equation} \label{eq:MeanToTarget}
\abs{m^\infty_J - x_{T,J}} \leq \dfrac{\mu\nu}{\tau}\,  
m^\infty_{1-\alpha(\delta), J}(\z),
\end{equation}
where the quantity $m^\infty_{1-\alpha(\delta), J}(\z)$ is finite under the 
assumption  $\delta(\z) \in \{-1,1\}$. Therefore, from 
bound^^>\eqref{eq:MeanToTarget}, we have that a vanishing penalization^^>$\nu$ 
leads to a relaxation of the mean to the target^^>$x_{T,J}$.

Therefore, looking at the variance with respect to the uncertainties $\z \in 
\mathbb R^{d_\z}$, we have for all $J \in \mathcal C$
\[
\textrm{Var}_{\z}(m^\infty_{J}(\z)) = \textrm{Var}_{\z}(m^\infty_J(\z) - 
x_{T,J}) = \mathbb E_{\z}[(m^\infty_J(\z) - x_{T,J})^2] - \mathbb 
E_{\z}[m^\infty_J(\z) - x_{T,J}]^2, 
\]
from which we get
\[
\textrm{Var}_{\z}(m^\infty_J(\z)) \le \mathbb E_{\z}[(m^\infty_J(\z) - 
x_{T,J})^2]
\le \left(\dfrac{\mu\nu}{\tau} \right)^2 \mathbb 
E_{\z}[m_{1-\alpha(\delta),J}^{\infty}(\z)]^2 \to 0
\]
for $\nu \to 0$. 

\item If we consider now $S(x)=\sqrt{x}$, from Jensen's inequality we have
\[
\intRp x^2 f^{\infty}_{J} (\z,x) \, dx \geq \lrp*{\intRp x f^{\infty}_{J} 
(\z,x) \, dx}^2,
\]
so that
\[
\intRp x(x-x_{T,J}) f^{\infty}_{J} (\z,x) \, dx \geq 
m^\infty_J(m^\infty_J-x_{T,J}).
\]
Therefore, we obtain the estimate
\[
\abs{m^\infty_J - x_{T,J}} \leq \dfrac{\mu\nu}{\tau}\,  
\dfrac{m^\infty_{1-\alpha(\delta), J}}{m^\infty_J},
\]
which again, for vanishing penalization $\nu$, implies that the mean reaches 
the target. Considering the variance with respect to the random variables $\z\in\mathbb R^{d_z}$, 
we obtain
\[
\textrm{Var}_{\z}(m^\infty_J(\z)) \le \mathbb E_{\z}[(m^\infty_J(\z) - 
x_{T,J})^2]
\le \left(\dfrac{\mu\nu}{\tau} \right)^2 \mathbb 
E_{\z}\left[\dfrac{m^\infty_{1-\alpha(\delta), J}}{m^\infty_J}\right]^2 \to 0
\]
for $\nu \to 0$. 
\end{itemize}
Hence, we argue that the introduced controls are capable of damping the 
variability due to the presence of uncertainties in the distribution of social 
contacts.

Furthermore in the case of zero diffusion case
$\sigma^2= 0$ we have 
\begin{equation}
\frac{d}{d\xi} E_J(\z,\xi) = \frac{1}{\epsilon\tau} \intRp B(\z,x)
\bracket{(x'_J)^2-x^2}f_J(\z,x,\xi)\, dx 
+ \frac{1}{\epsilon} \intRp ((x''_J)^2- x^2)f_J(\z,x,\xi)\, dx.
\end{equation}
In the limit $\epsilon\to 0$ and $t\to +\infty$ and with the scaled 
penalization $\kappa = \nu \tau$ we obtain
\[
\intRp \Phi^\delta(\z,x/m_{J}) x^{2-\alpha(\delta)} f^{\infty}_{J} (\z,x) \,dx 
= - 
\frac{\tau}{\nu}\intRp  S^2(x) x(x-x_{T,J}) f^{\infty}_{J} (\z,x) \, dx
\]
from which
\[
 \abs*{ \intRp S^2(x) (x^2-x x_{T,J} ) f^{\infty}_{J} (\z,x) \,dx } \leq 
 \dfrac{\mu \nu}{\tau}  m^\infty_{2-\alpha(\delta), J}(\z).
\]
\begin{itemize}
    \item Considering $S(x)\equiv1$, we get
    \[
    0 \le \abs{E^\infty(\z) - m^\infty(\z)\cdot x_{T,J}} \le 
    \dfrac{\mu\nu}{\tau}\, m^\infty_{2-\alpha(\delta), J}(\z),
    \]
    which gives the bound
    \[
    0 \le \abs{E^\infty(\z) - (m^\infty(\z))^2} \le \frac{\mu\nu}{\tau}\cdot 
    m_{2-\alpha(\delta),J}^\infty(\z),
    \]
    observing that in the limit $\nu \to 0$ we have $m^\infty(\z)\to x_{T,J}$.
    \item If we consider $S(x)=\sqrt{x}$, we have
    \[
    0 \le \abs{(m^\infty(\z))^3 - E^\infty(\z)\cdot x_{T,J}} \le 
    \dfrac{\mu\nu}{\tau}\, m^\infty_{2-\alpha(\delta), J}(\z),
    \]
    where again, in the limit $\nu\to 0^+$, we have $m^\infty(\z)\to x_{T,J}$.
\end{itemize}

Therefore, we can observe that the introduced controls push the energy $E^{\infty}(\z)$ towards the square of the mean number of contacts $m^{\infty}(\z)$. In other words, the steady state converges to a Dirac delta distribution centered at $x = x_{T,J}$.  

\subsection{Controlled kinetic epidemic model}

Once defined the control of the social dynamics, we can define a new kinetic 
epidemic model embedding the presence of non-pharmaceutical interventions. 
Following the discussions of Section \ref{subsect:controlled}, we combine the 
epidemic process with the controlled contact dynamics as 
\begin{equation}
\left\lbrace
\begin{aligned}
\pd{f_S(\z,x,t)}{t}  &= -K(f_S, f_I)(\z,x,t) + \frac{1}{\tau} Q_S(f_S)(\z,x,t) 
+ C_S(f_S)(\z,x,t),\\ 
\pd{f_E(\z,x,t)}{t}  &= K(f_S, f_I)(\z,x,t) - \zeta(x) f_E(\z,x,t) + 
\frac{1}{\tau} Q_E(f_E)(\z,x,t) + C_E(f_E)(\z,x,t),\\
\pd{f_I(\z,x,t)}{t}  &= \zeta(x) f_E(\z,x,t) - \gamma(x)f_I(\z,x,t) 
+\frac{1}{\tau} Q_I(f_I)(\z,x,t) + C_I(f_I)(\z,x,t),\\
\pd{f_R(\z,x,t)}{t}  &= \gamma(x) f_I(\z,x,t) + \frac1\tau Q_R(f_R)(\z,x,t) + 
C_R(f_R)(\z,x,t).
\end{aligned}
\right.
\label{eq:seir_control}
\end{equation}
As discussed in Section \ref{sect:2}, the transmission of the infection is governed 
by the local incidence rate $K(f_S,f_I)$ defined in \eqref{eq:K}, the 
thermalization of the distribution of social contacts in each compartment is 
given by $Q_J(f_J)$ together with the operators $C_J(f_J)$ defined in 
\eqref{eq:surr_FPC}. 

It is interesting to observe how, under the introduced scaling, the definition 
of non-pharmaceutical interventions acts at the same time scale of the epidemic 
dynamics. Hence, the equilibrium states of the dynamics of social contacts
result unaltered by the introduction of the control. This fact will be 
essential in the subsequent section to derive second order macroscopic models 
describing the evolution of the conserved moments of \eqref{eq:seir_control}.

\section{Observable effects of non-pharmaceutical interventions}\label{sect:4}

Epidemiological data are typically macroscopic quantities characterizing the 
evolution of a subset of the introduced compartments. In the following, we 
derive a macroscopic model which is consistent with the introduced kinetic 
epidemic model.

We recall here that in^^>\cite{Dimarco21,Dimarco22,Zanella21} one of the 
underlying assumptions was that the contact distribution of the population 
could be fruitfully estimated as an experimentally consistent Gamma 
distribution^^>\cite{Beraud15}. In this work, we put uncertainty precisely on 
the nature of the tail of the contact distribution, which in principle changes 
the characteristic of the related macroscopic system, thus changing also the 
efficacy of the containment strategies.

\subsection{Derivation of the macroscopic model}

Recalling that the operators $Q_J$ and $C_J$, coupled with no-flux boundary 
conditions, are mass-preserving, let us integrate 
system^^>\eqref{eq:seir_control} with respect to^^>$x$ to obtain
\begin{equation}
\left\lbrace
\begin{aligned}
\frac{d \rho_S(\z,t)}{dt}  &= - \beta 
m_S(\z,t)\rho_S(\z,t)m_I(\z,t)\rho_I(\z,t) ,\\ 
\frac{d \rho_E(\z,t)}{dt}  &= \beta m_S(\z,t)\rho_S(\z,t)m_I(\z,t)\rho_I(\z,t) 
- \zeta \rho_E(\z,t),\\
\frac{d \rho_I(\z,t)}{dt}  &= \zeta \rho_E(\z,t) - \gamma \rho_I(\z,t),\\
\frac{d \rho_R(\z,t)}{dt}  &= \gamma  \rho_I(\z,t),
\end{aligned}
\right.
\label{eq:mass}
\end{equation}
under the assumption on the local incidence rate \eqref{eq:Kalpha1}. In 
\eqref{eq:mass} we obtained a system for the evolution of the mass fractions. 
However, we can observe that the system is not closed like in the ones in the  
classical compartmental framework, since the evolution of $\rho_J(\z,t)$ depends 
on the evolution of the first order moment of the distribution functions 
$f_J(\z,x,t)$. The evolution of the momentum reads 
\[
\begin{split}
\dfrac{d}{dt} (\rho_S(\z,t)m_S(\z,t)) &= -\beta 
m_{2,S}(\z,t)\rho_S(\z,t)m_I(\z,t)\rho_I(\z,t) + \int_{\mathbb R^+} x 
C_S(f_S)(\z,x,t)dx, \\
\dfrac{d}{dt} (\rho_E(\z,t)m_E(\z,t)) &= \beta 
m_{2,S}(\z,t)\rho_S(\z,t)m_I(\z,t)\rho_I(\z,t) - \zeta m_E(\z,t)\rho_E(\z,t) + 
\int_{\mathbb R^+} x C_E(f_E)(\z,x,t)dx, \\
\dfrac{d}{dt} (\rho_I(\z,t)m_I(\z,t)) &= \zeta m_E(\z,t) \rho_E(\z,t) - \gamma 
m_I(\z,t)\rho_I(\z,t) + \int_{\mathbb R^+} x C_I(f_I)(\z,x,t)dx, \\
\dfrac{d}{dt} (\rho_I(\z,t)m_I(\z,t)) &= \gamma m_I(\z,t)\rho_I(\z,t) + 
\int_{\mathbb R^+} x C_R(f_R)(\z,x,t)dx, 
\end{split}
\]
where from \eqref{eq:CJ} we get
\[
\int_{\mathbb R^+} x C_J(f_J)(\z,x,t)dx = \int_{\mathbb R^+}   S^2(x)(x_T-x) 
f_J(\z,x,t)dx. 
\]
The hierarchical coupling of moments is a well-known problem in kinetic theory. 
The closure can, however, be obtained formally by
resorting to a limit procedure. Indeed, assuming that the 
time scale involved in the process of contact formation is $\tau \ll 1$, 
we obtain a fast thermalization of the contact distribution of agents with 
respect to the evolution of the epidemics. Therefore, for $\tau \ll 1$ the 
distribution function $f_J(\z,x,t)$ reaches fast the steady state equilibrium, 
which is a generalized Gamma distribution 
with mass fractions $\rho_J^\infty$ and local mean values $m_J^\infty$.

As observed in Remark \ref{rem:1}, the case in which $\delta(\z)$ is a discrete 
random variable such that $\delta(\z) \in \{-1,1\}$ is particularly interesting 
in the present modeling approach since the mean is conserved. In the following, 
we stick to this choice and we assume also 
\( \label{eq:delta}
\delta(\z) = 1 -2\z, \qquad \z\sim\text{Bernoulli}(p),
\)
such that
\[
\delta(\z) = 
\begin{cases}
-1 & \textrm{Prob}(\delta = -1) = p \\
1 & \textrm{Prob}(\delta = 1) = 1-p.
\end{cases}
\]
Under this assumption, we can express the second order moment of 
 the generalized Gamma distributions in terms of the mean 
\[
m_{2,J}^\infty(\z) = \int_{\mathbb R^+}x^2 f^\infty_J(\z,x)dx = 
\Lambda_\delta(\z) (m_J^\infty(\z))^2, \qquad \Lambda_\delta(\z) = \lrp*{ 
\dfrac{\lambda + \delta(\z)}{\lambda} }^{\delta(\z)}, 
\]
where we recall that we fixed $\lambda = \mu/\sigma^2$. Therefore, at the 
macroscopic level, we obtain the following system of equations for the time 
evolution of the first order moments in each compartment
\begin{equation}
\left\lbrace
\begin{aligned}
\frac{d m_S(\z,t)}{dt}  &= -\beta \lrp{\Lambda_\delta(\z) -1 } m_S^2(\z,t) 
m_I(\z,t) \rho_I(\z,t) 
                            + G_S(f^\infty_S)(\z,t) \\
\frac{d m_E(\z,t)}{dt}  &= \beta \frac{m_S(\z,t)\rho_S(\z,t) 
m_I(\z,t)\rho_I(\z,t)}{\rho_E(\z,t)} 
                            \lrp{\Lambda_\delta(\z) m_S(\z,t) - m_E(\z,t)}
                            +G_E(f^\infty_E)(\z,t) \\
\frac{d m_I(\z,t)}{dt}  &= \zeta \frac{\rho_E(\z,t)}{\rho_I(\z,t)} (m_E(\z,t) - 
m_I(\z,t)) 
                            + G_I(f^\infty_I)(\z,t) \\
\frac{d m_R(\z,t)}{dt}  &=  \gamma  \frac{\rho_I(\z,t)}{\rho_R(\z,t)} 
(m_I(\z,t) - m_R(\z,t))
                            + G_R(f^\infty_R)(\z,t).
\end{aligned}
\right.
\label{eq:seirAme2}
\end{equation}

In \eqref{eq:seirAme2} the terms $G_J(f_J)$, $J \in \C$, embed the action of the control at 
the level of the mean number of social contacts and read 
\begin{equation}
   G_J(f^\infty_J)(\z,t) = \frac{1}{\nu \rho_J(\z,t)} \intRp S^2(x)(x_{T,J}-x) 
   f^\infty_J(\z,x) dx.
    \label{eq:CJfinfty}
\end{equation}
We observe now that \eqref{eq:mass} and \eqref{eq:seirAme2} describe in closed 
form the time evolution of an epidemic where the transition between 
compartments depend on the mean number of social contacts in the population.

In particular, in the cases $S^2(x)\equiv 1$ and $S^2(x) = x$,  we have
\[
G_J(f_J^\infty)(\z,t) = 
\begin{cases}\vspace{0.25cm}
\dfrac{1}{\nu}\left[ x_T  - m_J(\z,t)\right], & S^2(x) \equiv 1 \\
\dfrac{m_J(\z,t)}{\nu}\left[ x_T - \Lambda_\delta(\z) m_J(\z,t)\right] & S^2(x) =x, 
\end{cases}
\]
For small penalization of the control $\nu\to 0^+$, the mean number of 
connections stabilizes towards the values 
\[
m_J^\infty(\z) =  
\begin{cases}
x_T & S^2(x) \equiv 1 \\
\dfrac{x_T}{\Lambda_\delta(\z)} & S^2(x) = x. 
\end{cases}
\] 
Therefore, a selective strategy may outperform the uniform one depending on 
the value of $\Lambda_\delta(\z)$. We observe that, for vanishing penalizations, the expected number of connections in the compartment $J \in \C$ are such that $\mathbb E_\z[m_J^\infty(\z)] < x_T$ if $p<1/2$, indeed exploiting the information in \eqref{eq:delta} we get
\[
\mathbb E_\z[\Lambda_\delta(\z)] = \dfrac{\lambda+1-2p}{\lambda}>1. 
\]
%

\section{Numerical examples}\label{sec:numerics}
In this section, we present several numerical results. We first construct an 
implicit structure preserving (SP) method \cite{Pareschi18, Zanella20} with a 
stochastic-Galerkin approach \cite{Xiu10, Dimarco22_2, Zhu17} for 
system^^>\eqref{eq:seir_control}. This kind of methods are spectrally accurate 
in the space of the random parameters under suitable regularity assumptions. For a survey on available methods for the uncertainty quantification of kinetic models we mention \cite{Pareschi21} and the references therein.   In 
particular, we study the influences of the uncertainties in the spreading of an 
epidemics and the capability of the designed control strategies in reducing 
both the peak of the epidemics and the variability of the results given by the 
random parameters. 

Furthermore, we consider the macroscopic system of ODEs \eqref{eq:mass}-\eqref{eq:seirAme2} 
and we estimate relevant parameters characterizing  non-pharmaceutical interventions based on real epidemiological data. We first estimate 
the relevant epidemiological parameters thanks to the dataset of the John 
Hopkins University\footnote{https://github.com/CSSEGISandData/COVID-19\quad 
Last accessed: 26th September 2022.}. Hence, we evaluate the impact of 
different control strategies during the first wave of infection in Italy.

\subsection{Stochastic Galerkin methods} \label{sec:sGmethod}
\begin{figure}[t]
	\centering
    \includegraphics[width = 0.45\linewidth]{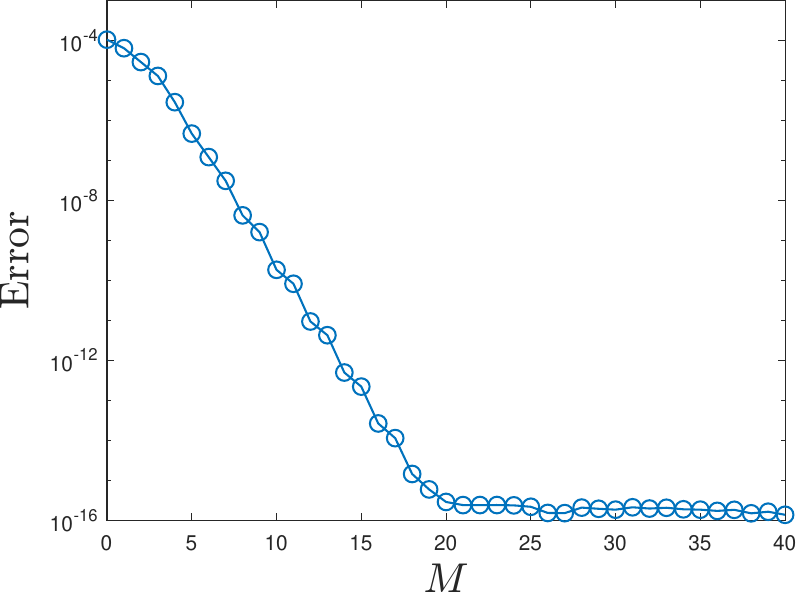}
	\caption{Convergence of the $L^2$ error of the first order moment with respect to a 
	reference solution computed with $M=40$ at fixed time $T=1$. We choose $\Delta x=0.02$ in the interval $[0,500]$, and $\Delta t=0.1$ with $\tau=10^{-5}$. 
	The uncertain parameter is $\delta(\z)=\z$ with 
	$\z\sim\mathcal{U}([-1,1])$. Initial conditions given by \eqref{eq:test0}.}
	\label{fig:figure1}
\end{figure}

In order to solve numerically system \eqref{eq:seir_control}, let us rewrite it 
in vector form 
\( \label{eq:sys_vector2}
\frac{\partial \mathbf{f}}{\partial t} (\z,x,t) = \mathbf{{P}}\lrp{ x, 
\mathbf{f} (\z,x,t) } + \frac{1}{\tau} \mathbf{{Q}}\lrp{ \mathbf{f} (\z,x,t) } 
+ \mathbf{{C}}\lrp{ \mathbf{f} (\z,x,t) },
\)
where $\mathbf{f}=\{f_J\}_J$, $\mathbf{{Q}}=\{Q_J\}_J$, 
$\mathbf{{C}}=\{C_J\}_J$, $J=\{S,E,I,R\}$, and $\mathbf{{P}}$ is the vector 
whose components are the transitions rates between the compartments.

Stochastic Galerkin (sG) methods are based on the approximation of the solution 
$\mathbf{f} (\z,x,t)$ on a set of polynomials $\{\Psi_h(\z)\}_{h=0}^M$ of 
degree less or equal to $M\in\mathbb{N}$, orthonormal with respect to the 
distribution of the random parameters, such that
\[
\mathbf{f} (\z,x,t) \approx \mathbf{f}^M (\z,x,t) = \sum_{h=0}^M 
\hat{\mathbf{f}}_h (x,t) \Psi_h(\z).
\]
The polynomials are chosen following the so-called Wiener--Askey scheme 
\cite{Xiu10, Xiu02}. In the previous relation, we denote by $\hat{\mathbf{f}}_h 
(x,t)=\{\hat{f}_{h,J}(x,t)\}_J$ the projections of the solution along the 
linear space generated by the polynomial of degree $h$
\[
\hat{\mathbf{f}}_h (x,t) = \int_\Omega \mathbf{f} (\z,x,t) \Psi_h(\z) p(\z) d\z 
\coloneqq \mathbb{E}_{\z} [\mathbf{f} (\z,x,t) \Psi_h(\z)],
\]
where we denote by $\Omega\subseteq\R^{d_\z}$ the space of the random parameters.

We discretize the time domain $[0,T]$ with a time step of size $\Delta t>0$ and 
we denote by $\mathbf{f}^n(x)$ an approximation of $\mathbf{f}(x,t^n)$ with 
$t^n=n\Delta t$. The first order time splitting method reads:
\begin{equation}
\text{Contact \& control dynamics: }
\left\lbrace
\begin{aligned}
&\frac{\partial\mathbf{f}^*}{\partial t} =\frac{1}{\tau} \mathbf{{Q}}\lrp{ 
\mathbf{f}^* } + \mathbf{{C}}\lrp{ \mathbf{f}^* }, \\ 
&\mathbf{f}^*(\z,x,0) = \mathbf{f}^n(\z,x),
\end{aligned}
\right.
\label{eq:FP_split_c}
\end{equation} 
\begin{equation}
\text{Epidemic exchange: }
\left\lbrace
\begin{aligned}
&\frac{\partial\mathbf{f}^{**}}{\partial t} =\mathbf{{P}}\lrp{x, 
\mathbf{f}^{**} }, \\ 
&\mathbf{f}^{**}(\z,x,0) = \mathbf{f}^*(\z,x,\Delta t).\\
\end{aligned}
\right.
\label{eq:Epidemic_split_c}
\end{equation}
\reply{We plug $\mathbf{f}^M$ into 
\eqref{eq:FP_split_c}--\eqref{eq:Epidemic_split_c} and we project against 
$\Psi_h(\z) p(\z) d\z$ on $\Omega$ for each $h=0,\dots,M$.} Hence, we 
end with two systems of $M+1$ vector equations for the coefficients of the 
expansion. 

The sG reformulation of the contact embedding the control dynamics reads
\begin{align}\label{eq:FPsG}
\frac{\partial \hat{f}^*_{h,J} }{\partial t} (x,t) = &  
\frac{\partial}{\partial x} \sum_{k=0}^M \hat{f}^*_{k,J} (x,t) 
\!\int_\Omega \lrp*{ \frac{\mu 
x^{1-\alpha(\delta(\z))}}{2\delta(\z)}\lrp*{\!\lrp*{\frac{x}{m_J(\z,t)}}^{\delta(\z)} -1} 
+ \frac{S^2(x)}{\nu} (x-x_{T,J})\!} \Psi_k(\z) \Psi_h(\z) p(\z) d\z\notag 
\\
& + \frac{\partial^2}{\partial x^2} \sum_{k=0}^M \hat{f}^*_{k,J} (x,t) 
\int_\Omega \frac{\sigma^2}{2} x^{2-\alpha(\delta(\z))}  \Psi_k(\z) \Psi_h(\z) 
p(\z) d\z.
\end{align}
We discretize \eqref{eq:FPsG} with a central finite differences approach and 
we apply a fully-implicit-in-time scheme following the construction presented in \cite{Dimarco22_2, Zanella20}. 

The epidemic exchange system is
\begin{equation}
\label{eq:hatfJ_time}
\left\lbrace
\begin{aligned}
\pd{\hat{f}_{h,S}(x,t)}{t}  &= - \beta x \sum_{k=0}^M \hat{f}_{k,S}(x,t) 
\int_\Omega m_{I}^{M}(\z,t) \rho_{I}^{M}(\z,t) \Psi_k(\z) \Psi_h(\z) p(\z) 
d\z ,\\ 
\pd{\hat{f}_{h,E}(x,t)}{t}  &= \beta x \sum_{k=0}^M \hat{f}_{k,S}(x,t) 
\int_\Omega m_{I}^{M}(\z,t) \rho_{I}^{M}(\z,t) \Psi_k(\z) \Psi_h(\z) p(\z) 
d\z - \zeta(x) \hat{f}_{h,E}(x,t),\\
\pd{\hat{f}_{h,I}(x,t)}{t}  &= \zeta(x) \hat{f}_{h,E}(x,t) - \gamma(x) 
\hat{f}_{h,I}(x,t),\\
\pd{\hat{f}_{h,R}(x,t)}{t}  &= \gamma(x) \hat{f}_{h,I}(x,t),
\end{aligned}
\right.
\end{equation}
where 
\[
m_{I}^{M}(\z,t) \rho_{I}^{M}(\z,t) = \intRp x f_{I}^{M}(\z,x,t) dx.
\]
System \eqref{eq:hatfJ_time} is then integrated through a first order Euler method. 

To show the spectral convergence property of the designed sG method, 
we consider a contact dynamics in the uncontrolled scenario, i.e. with 
$S(x)=0$, of a generic compartment^^>$J$, that is, we take a single component 
of \eqref{eq:FPsG}. We compute a reference solution with $N=25001$ grid points 
of size $\Delta x = 0.02$ in the $x$-domain $[0,500]$, $\Delta t=0.1$, 
$\tau=10^{-5}$ and sG expansion up to order $M=40$. We fix the parameters as 
$\mu = 0.5$, $\sigma^2 = 0.1$, being $\lambda=\mu/\sigma^2$, and we consider a 
one-dimensional uncertainty in a way that $\delta(\z)=\z$ with 
$\z\sim\mathcal{U}([-1,1])$. Since the distribution of $\z$ is uniform, we 
consider Legendre polynomials. The initial distribution is a deterministic Gamma
\( \label{eq:test0}
f^0_J(x)= \frac{x^{\lambda-1} e^{\lambda x / m_J} 
(\lambda/m_J)^\lambda}{\Gamma(\lambda)}
\)
with $m_J=10$. Then, we compute the $L^2$ error on the first order moment of 
the distribution at fixed time $T=1$ for increasing $M$. 

In Figure^^>\ref{fig:figure1}
, we may observe the decay of the numerical error in 
the space of the random parameter as the order of accuracy increases. We 
observe that we reach essentially the machine precision within a finite order 
$M$.

\subsection{Test 1: Uncontrolled model} \label{sec:uncontrolled}
\begin{figure}[t]
	\centering
	\includegraphics[width = 0.45\linewidth]{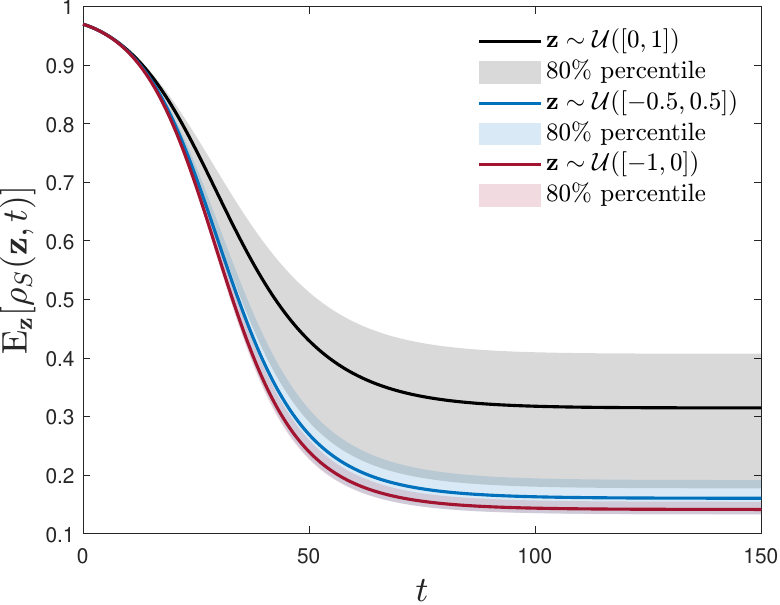}
	\includegraphics[width = 0.45\linewidth]{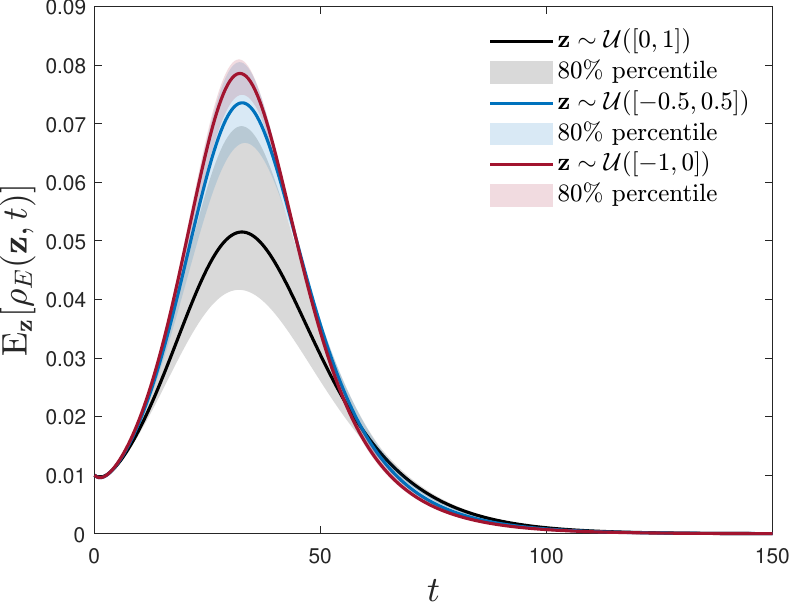}
	\includegraphics[width = 0.45\linewidth]{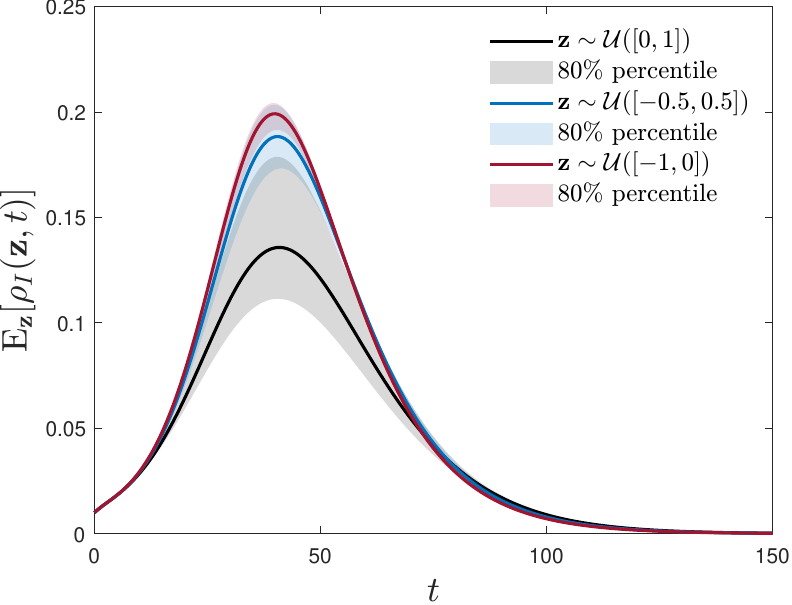}
	\includegraphics[width = 0.45\linewidth]{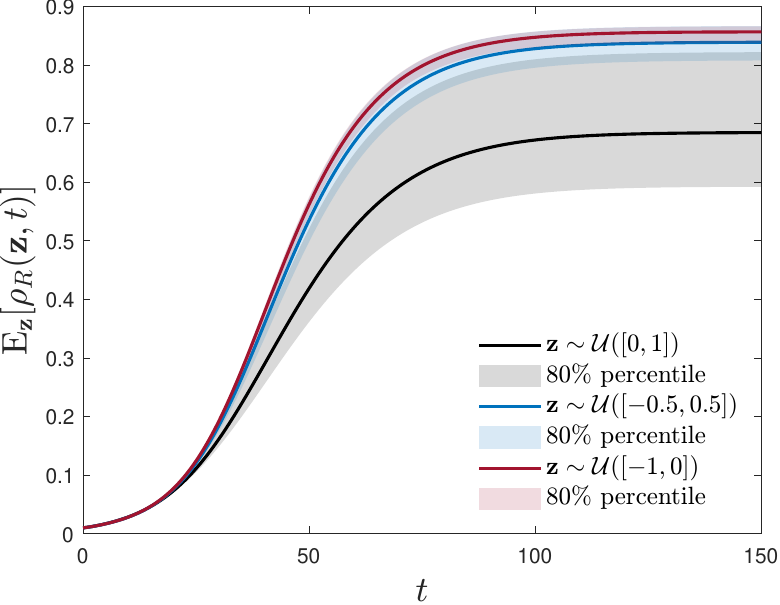}
	\caption{\small{\textbf{Test 1.} Expectations 
	of the masses $\rho_J(\z,t)$ for every compartment $J\in\{S,E,I,R\}$, as a 
	function of the time. We compare different scenarios corresponding to the 
	choices $\z\sim\mathcal{U}([0,1])$ (black), $\z\sim\mathcal{U}([-0.5,0.5])$ 
	(blue) and $\z\sim\mathcal{U}([-1,0])$ (red). We choose $\Delta x=0.02$ in 
	the interval $[0,500]$, $\Delta t=0.1$ with $T=150$ and $\tau=10^{-5}$. The 
	sG expansion is of order $M=5$. Initial conditions given by 
	\eqref{eq:test1_init}.}}
	\label{fig:figure2}
\end{figure}
In this section, we focus on the uncontrolled scenario, i.e., system 
\eqref{eq:sys_vector2} with $S(x)=0$. We fix the parameters as $\beta = 
0.0025$, $\gamma = 0.1$, $\zeta = 0.3$, $\mu = 0.5$, $\sigma^2 = 0.1$, with 
$\lambda=\mu/\sigma^2$, we consider a one-dimensional uncertainty in a way that 
$\delta(\z)=\z$ and we investigate the behavior of the model for a uniform 
random variable $\z$ with different support. The $x$-domain is $[0,500]$, 
discretized with $N=25001$ grid points of size $\Delta x = 0.02$, the time 
domain $[0,150]$ is discretized with the time step $\Delta t=0.1$; the scale 
parameter is $\tau=10^{-5}$. We fix the sG expansion up to order $M=5$ in all 
the simulations. The initial conditions for the $f^0_J(x)$ are deterministic 
Gamma distributions
\( \label{eq:test1_init}
f^0_J(x)=\rho^0_J \, \frac{x^{\lambda-1} e^{\lambda x / m^0_J} (\lambda / 
m^0_J)^\lambda}{\Gamma(\lambda)}
\)
with $\rho^0_S=0.97$, $\rho^0_E=\rho^0_I=\rho^0_R=0.01$ and $m^0_J=10$ for 
every compartment $J$.

In Figure^^>\ref{fig:figure2}
we show the time evolution of the masses of 
the compartments for different choices of the random parameter, namely:
\begin{enumerate}[a)] 
    \item \label{list:a} $\z\sim\mathcal{U}([0,1])$ (black);
    \item \label{list:b} $\z\sim\mathcal{U}([-0.5,0.5])$ (blue);
    \item \label{list:c} $\z\sim\mathcal{U}([-1,0])$ (red).
\end{enumerate}
We observe that the choice \ref{list:c} is associated to a contact equilibrium 
with fat tails, indicating that there exists a higher probability that agents 
possess a great number of contacts. Indeed we observe that this choice 
generates at the equilibrium the smallest number of Susceptible and the highest 
number of Removed with respect to the other ones, indicating that the epidemics 
has spread more. Moreover, note also how the peaks of the Infected and Exposed 
are above the others. The choice \ref{list:b}, associated to contact 
equilibrium with both fat  and slim tails, exhibits an intermediate behavior 
with respect to \ref{list:a}, which is associated to slim tails, and 
\ref{list:c}, as expected.

\subsection{Test 2: Consistency of the macroscopic limit}

\begin{figure}
	\centering
	\includegraphics[width = 0.35\linewidth]{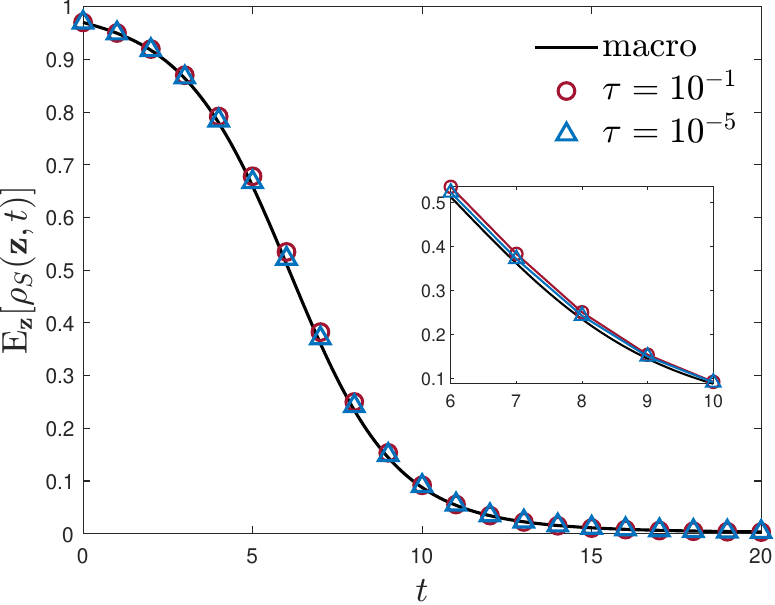}\qquad
	\includegraphics[width = 0.35\linewidth]{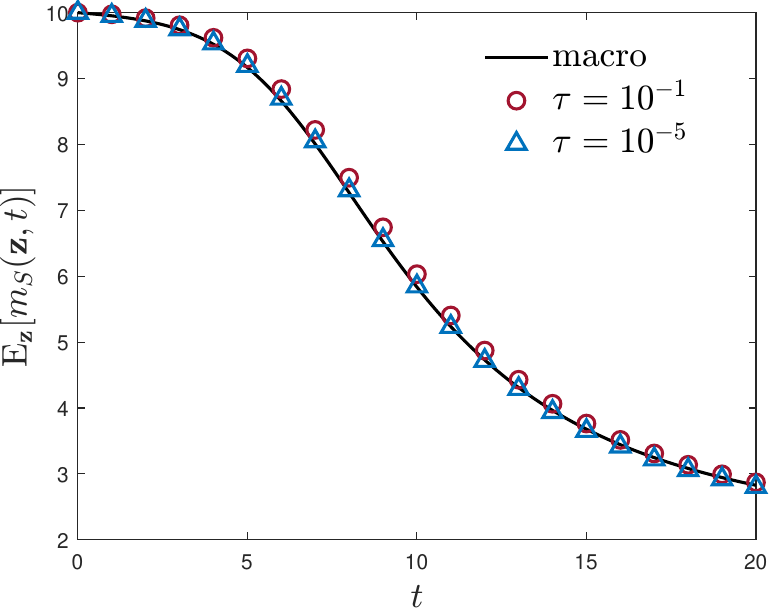}\\
	\includegraphics[width = 0.35\linewidth]{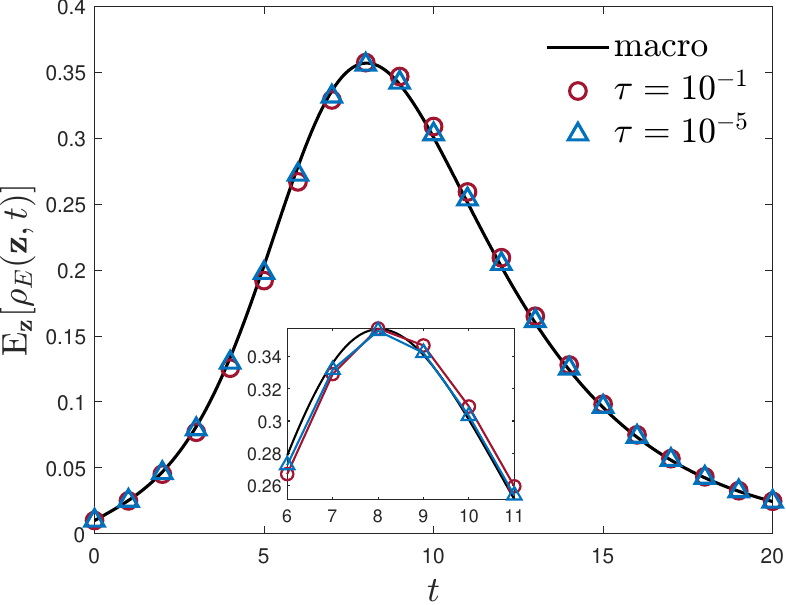}\qquad
	\includegraphics[width = 0.35\linewidth]{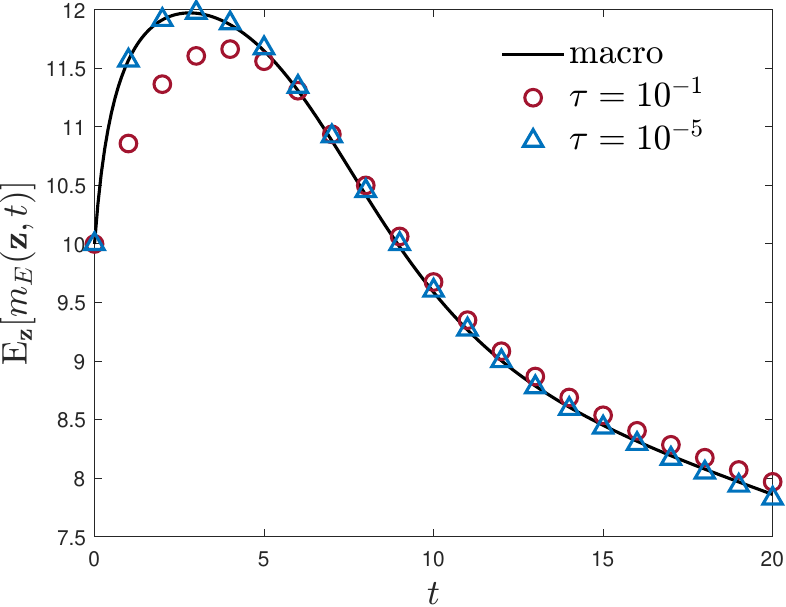}\\
	\includegraphics[width = 0.35\linewidth]{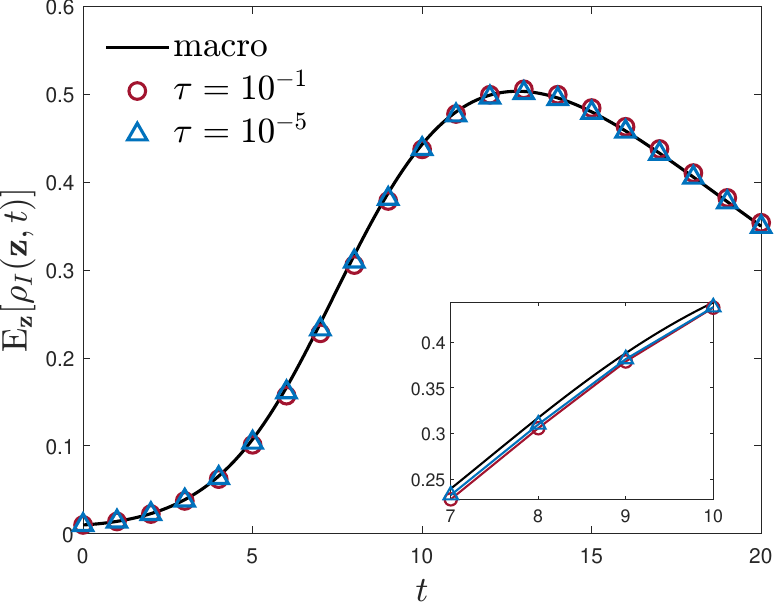}\qquad
	\includegraphics[width = 0.35\linewidth]{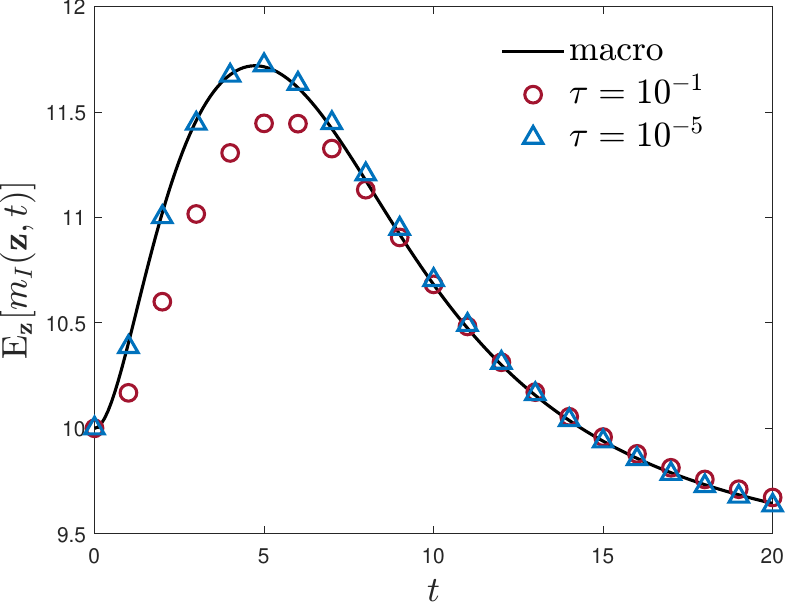}\\
	\includegraphics[width = 0.35\linewidth]{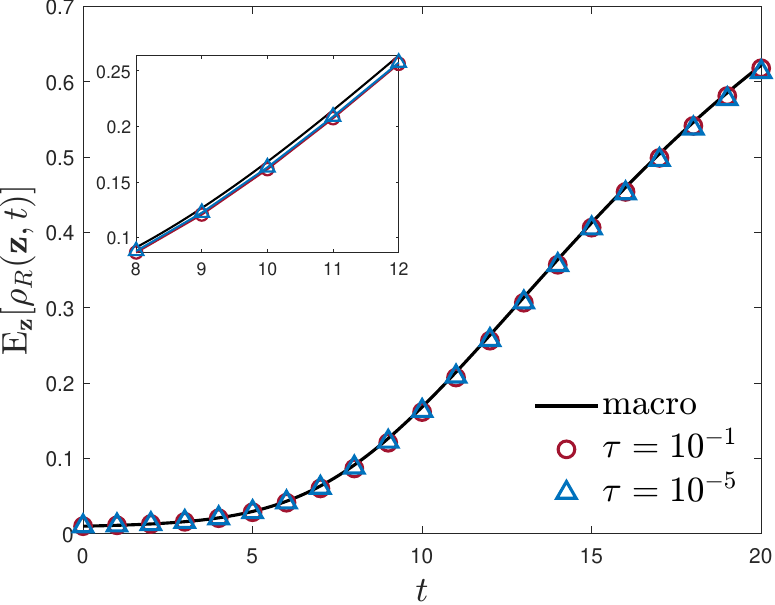}\qquad
	\includegraphics[width = 0.35\linewidth]{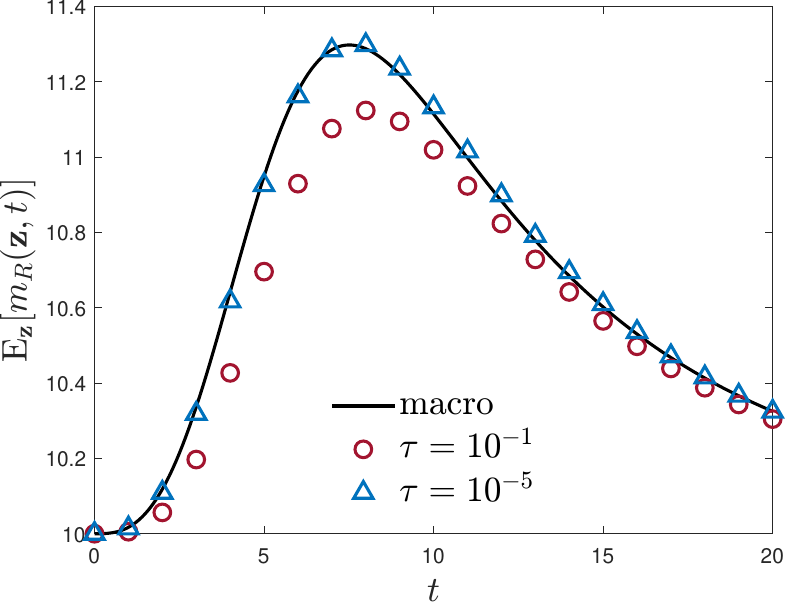}
	\caption{\small{\textbf{Test 2}. Time evolution of the mass 
	fractions (left column) and mean values (right column) obtained from the 
	integration of equation \eqref{eq:sys_vector2} with $\delta(\z)=1-2\z$, 
	$\z\sim\textrm{Bernoulli(p)}$, for $\tau=10^{-1},\,10^{-5}$, together with 
	the evolution of the mass fraction and mean values of the macroscopic model 
	\eqref{eq:mass}-\eqref{eq:seirAme2}, in the uncontrolled scenario. In both 
	cases, we fix $p=1/2$ and the epidemiological parameters as in Table 
	\ref{tab:fittedparam}. Kinetic equations solved with $\Delta x=0.02$ in the 
	interval $[0,500]$, and $\Delta t=0.1$ with $T=20$. Initial distribution as 
	in \eqref{eq:test1_init}.}}
	\label{fig:figure3}
\end{figure}

We consider the coupled system^^>\eqref{eq:mass}--\eqref{eq:seirAme2} with the 
underlying assumption that the random variable^^>$\z$ follows a Bernoulli 
distribution of parameter^^>$p$ as in \eqref{eq:delta}. In the following we will fix $p = 1/2$.
We numerically check the consistency of the derived macroscopic closure of the 
kinetic model which leads to the system^^>\eqref{eq:mass}--\eqref{eq:seirAme2} 
in the limit $\tau\to0^+$. We solve the coupled ODEs with a 
fourth-order Runge--Kutta method with $\Delta t=0.05$, the kinetic system 
\eqref{eq:sys_vector2} is solved with the same discretization described in 
Section \ref{sec:uncontrolled}, with the initial conditions 
\eqref{eq:test1_init}. The epidemiological parameters are summarized in Table 
\ref{tab:fittedparam}.
In Figure^^>\ref{fig:figure3}
we observe that smaller values of the time scale $\tau$ 
corresponds to better time-by-time accordance between the kinetic equations and the macroscopic model.

\subsection{Test 3: Controlled model and uncertainty damping}
\begin{figure}
	\centering
	\includegraphics[width = 0.45\linewidth]{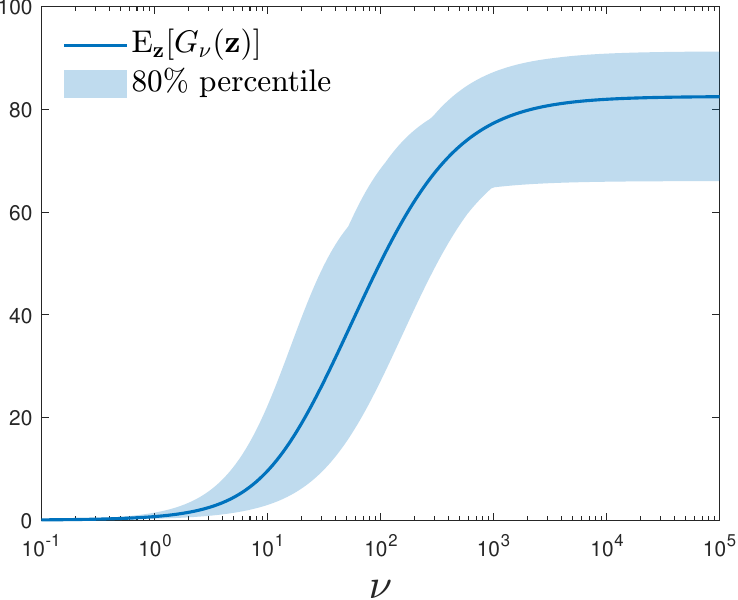}
	\includegraphics[width = 0.45\linewidth]{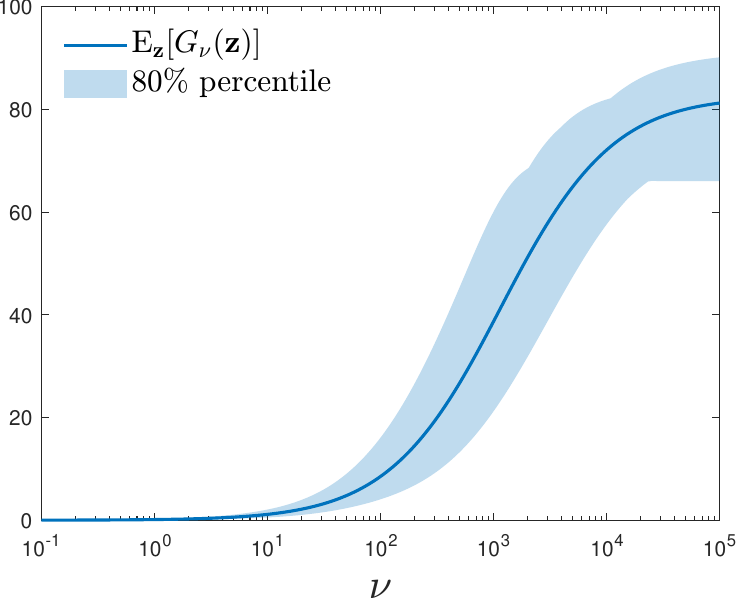}
	\includegraphics[width = 0.45\linewidth]{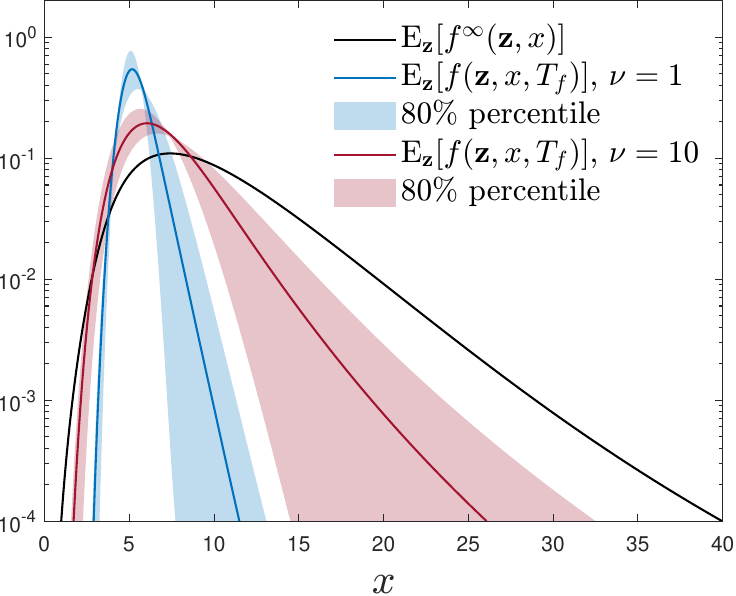}
	\includegraphics[width = 0.45\linewidth]{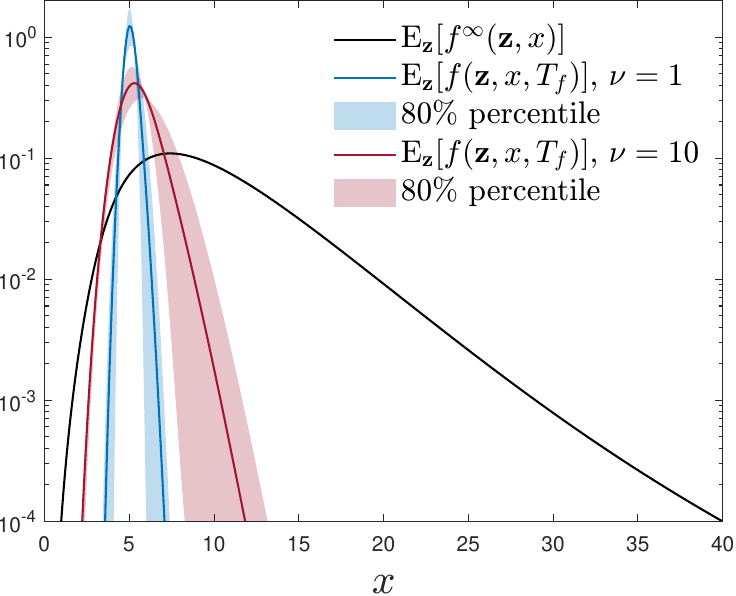}
	\caption{\small{\textbf{Test 3.} Top 
	row: expectation of $G_\nu(\z)$ defined in \eqref{eq:G_nu} with 
	$x_{T,J}=5$, versus the penalization $\nu$, for $S(x)=1$ (left) and 
	$S(x)=\sqrt{x}$ (right). Bottom row: details of the expected values of the 
	distributions: the black line represents the uncontrolled distribution at 
	the equilibrium, the blue and red lines are the controlled distribution for 
	$\nu=1,10$, at the fixed time \reply{$T_f=1$}, for the selective functions $S(x)=1$ 
	(left) and $S(x)=\sqrt{x}$ (right). In all the simulations we choose 
	$\Delta x=0.02$ in the interval $[0,500]$, $\Delta t=0.1$ with \reply{$T_f=1$} and 
	$\tau=10^{-5}$. The sG expansion is $M=5$, the uncertain parameter is 
	$\delta(\z)=\z$ with $\z\sim\mathcal{U}([-1,1])$. Initial conditions given 
	by \eqref{eq:test0}.}}
	\label{fig:figure4}
\end{figure}

\begin{figure}
	\centering
	\includegraphics[width = 0.35\linewidth]{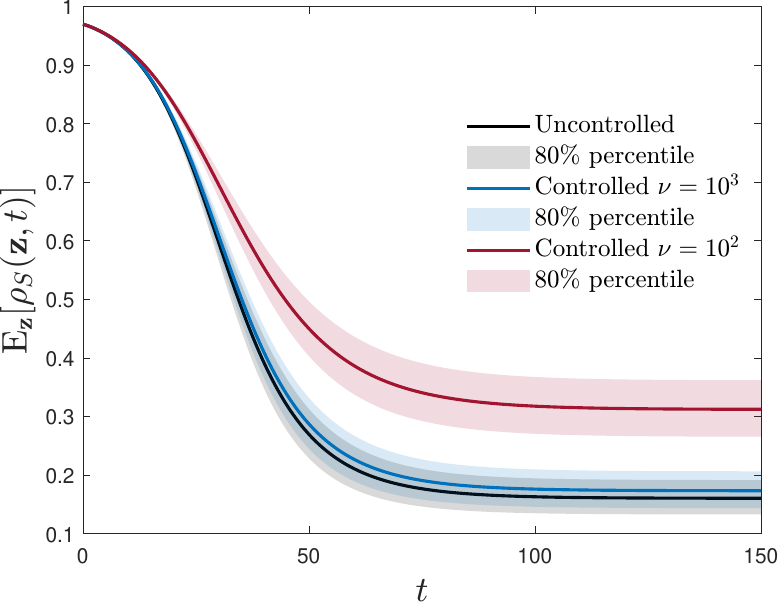}\qquad
	\includegraphics[width = 0.35\linewidth]{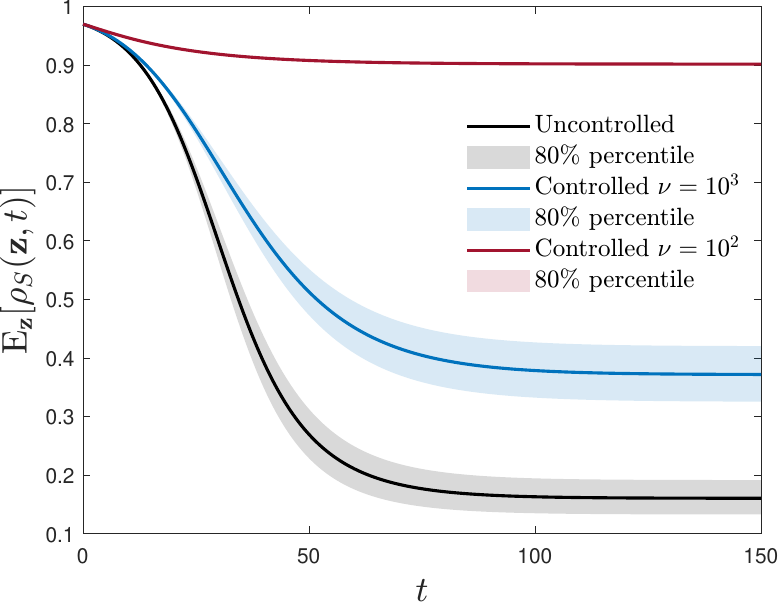}\\
	\includegraphics[width = 0.35\linewidth]{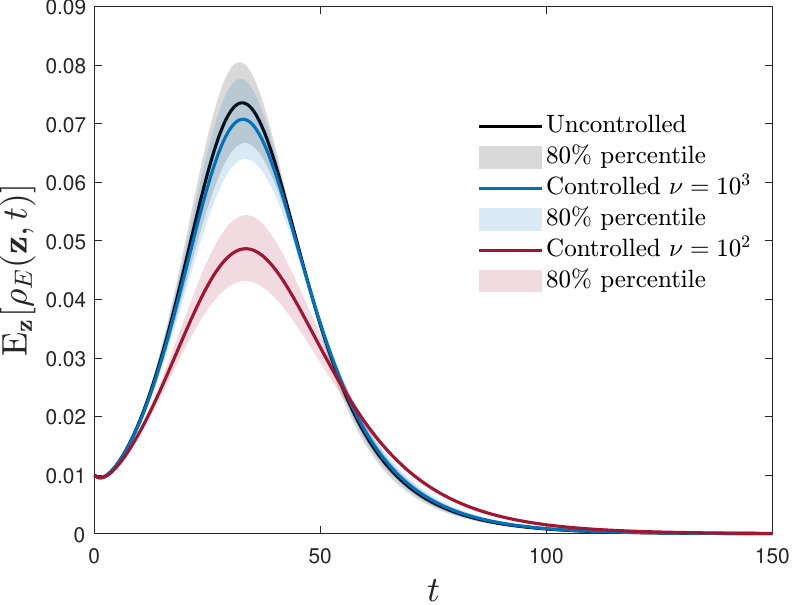}\qquad
	\includegraphics[width = 0.35\linewidth]{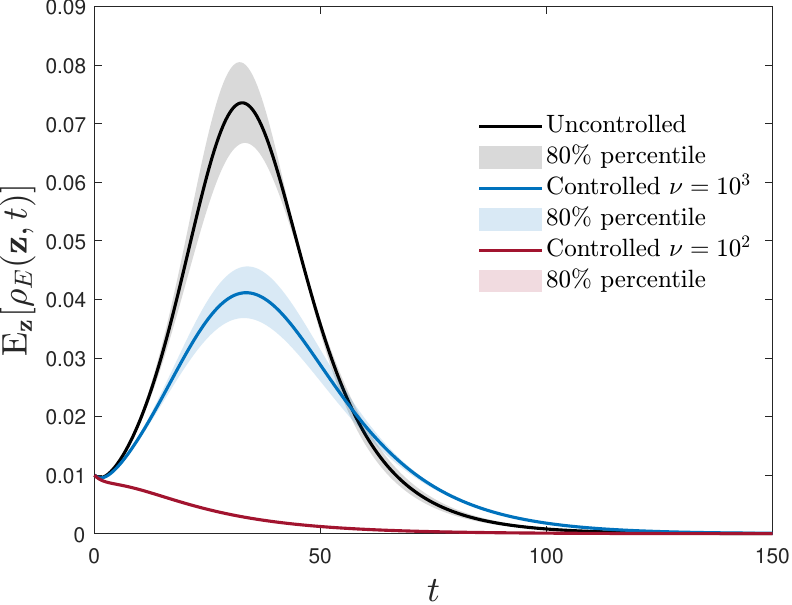}\\
	\includegraphics[width = 0.35\linewidth]{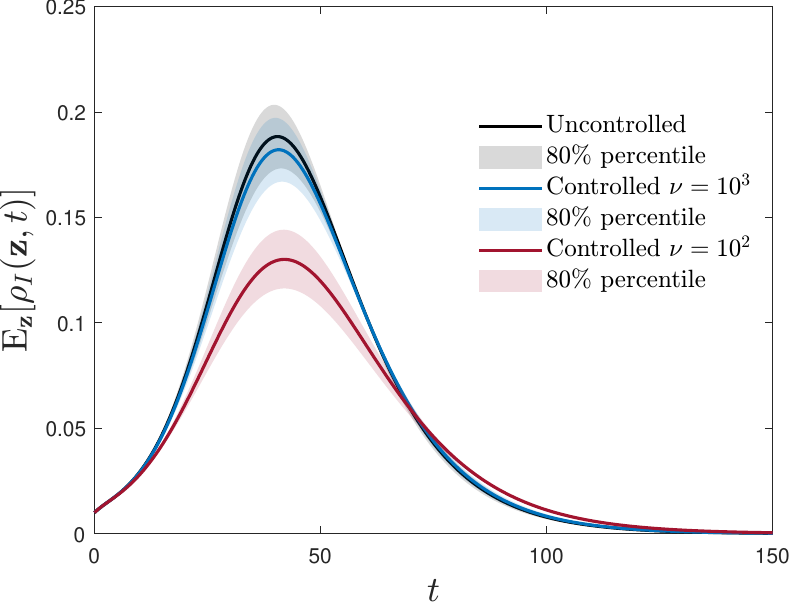}\qquad
	\includegraphics[width = 0.35\linewidth]{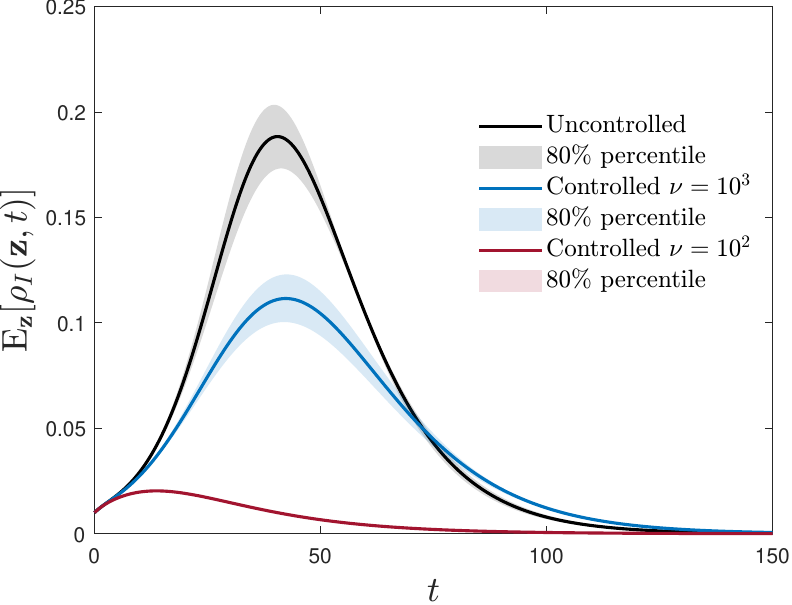}\\
	\includegraphics[width = 0.35\linewidth]{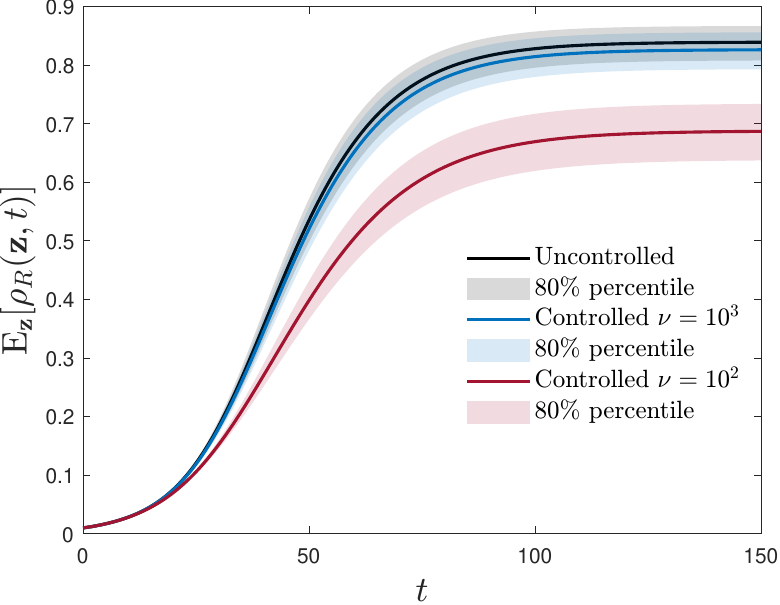}\qquad
	\includegraphics[width = 0.35\linewidth]{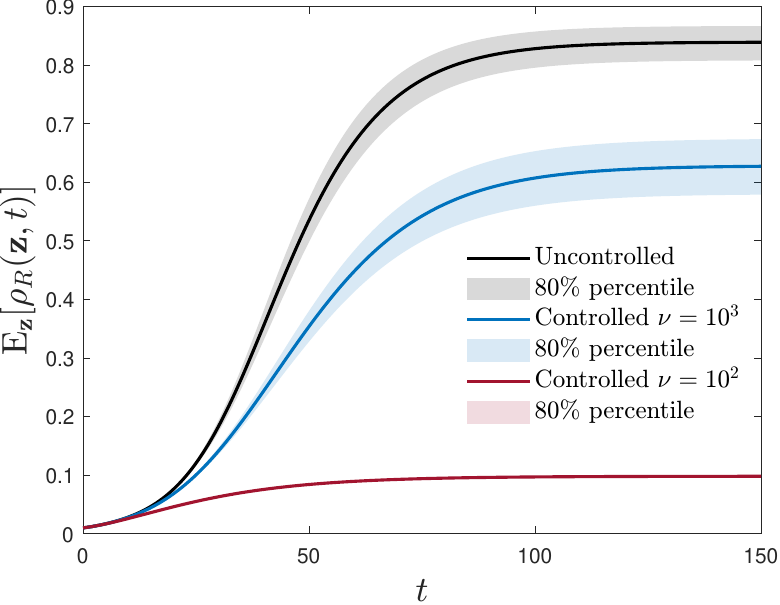}
	\caption{\small{\textbf{Test 3.} Time evolution 
	of the expectations of the masses $\rho_J(\z,t)$ for every compartment 
	$J=\{S,E,I,R\}$. The black line is the uncontrolled scenario, the red and 
	blue lines are the controlled time evolution for $\nu=10^2,10^3$ 
	respectively, for the selective functions $S(x)=1$ (left column) and 
	$S(x)=\sqrt{x}$ (right column). We choose $\delta(\z)=\z$ with 
	$\z\sim\mathcal{U}([-0.5,0.5])$, $\Delta x=0.02$ in the interval $[0,500]$, 
	and $\Delta t=0.1$ with $T=150$ and $\tau=10^{-5}$. The sG expansion is of 
	order $M=5$. Initial conditions are given in \eqref{eq:test1_init}.}}
	\label{fig:figure5}
\end{figure}

Let us consider now the controlled model. We concentrate 
first on the contact dynamics of a single generic compartment $J$ without 
epidemic exchange, i.e. a component of \eqref{eq:FPsG}. We are indeed 
interested in evaluating the effectiveness of the designed control in reducing 
the tails of the distributions and damping the uncertainties of the system. The 
parameters, the space and time discretization and the initial conditions are 
chosen as in Section \ref{sec:sGmethod}. We fix $M=5$ and $x_{T,J}=5$.

We choose two different selective functions, in the first case we assume $S(x)\equiv 1$, corresponding to a control that is uniform over the population being independent from the number of contacts. We consider then the selective  case with $S(x)=\sqrt{x}$, the resulting control has a stronger impact on agents  with a higher numbers of contacts. To quantify the effectiveness in reducing uncertainty of the adopted control strategy, we define an index that measures the distance from the target $x_T$ and the variability at a given time \reply{$T_f > 0$ }
\cite{Medaglia22}
\( \label{eq:G_nu}
G_\nu(\z)=\intRp (x-x_T)^2 f_J(\z,x,\reply{T_f}) dx. 
\)
On the top row of
Figure^^>\ref{fig:figure4},
we show the expectation of 
$G_\nu(\z)$ versus the penalization coefficient $\nu$, for the chosen selective 
functions. We observe that the control $S(x)=\sqrt{x}$ is more efficient than 
the uniform selection, in the sense that reduces more both the variability and 
the distance from the target for a fixed penalization^^>$\nu$, as discussed in 
Section \ref{sec:damping}.

On the bottom row of
Figure^^>\ref{fig:figure4},
we display in 
semilogarithmic scale the expectation of the uncontrolled distribution 
$f^\infty(\z,x)$ at the equilibrium, together with the expectations of the 
numerical solution of \eqref{eq:FPsG} at the fixed time \reply{$T_f=1$}, for 
penalizations $\nu=1,10$. Note how the introduced control is capable to change the behavior of the tails of the distribution.

Then, we consider the full model \eqref{eq:sys_vector2} with epidemic exchange. 
In particular, we are interested in understanding whether the control on the 
contact dynamics is able to reduce the spreading of the epidemics and the 
variability due to the uncertain parameter. To this end, we consider the 
computational setting of Section \ref{sec:uncontrolled} with 
$\z\sim\mathcal{U}([-0.5,0.5])$, \reply{$T_f=150$} and the selective functions 
$S(x)=1,\,\sqrt{x}$.
In Figure^^>\ref{fig:figure5}
we compare the 
time evolution of the expectations of the masses $\rho_J(\z,t)$ in the 
uncontrolled scenario (black) and under the action of the control with 
$\nu=10^3,10^2$ (blue and red), for all the compartments. We observe that the 
control is able to increase the fraction of Susceptible (first row) at the 
equilibrium and to reduce the Removed (fourth row), but also to dampen the 
peaks of Exposed (second row) and Infected (third row), meaning that the 
epidemics has spread less. As expected, with a fixed penalization, the 
selective control $S(x)=\sqrt{x}$ is more efficient than the uniform one, and 
it is also capable of reducing the uncertainties on the results, as we can 
notice from the right column, red lines, of
Figure^^>\ref{fig:figure5}

\subsection{Test 4: A data-oriented approach}
As remarked at the beginning of the section, and following the approach 
proposed in^^>\cite{Dimarco22}, we will focus on the first wave of the 
SARS-CoV-2 epidemic during the first half of 2020, particularly in the case of 
Italy. There, the first detected case was on January, 30th, while the first 
containment measures were applied on March, 9th.

%

\subsubsection{Test 4a: Calibration of the model}
The first step of the calibration is to estimate the unknown epidemiological 
parameters in the unconstrained regime, assuming that no restriction on the 
number of contacts was having place, which translates into having 
$G_J(f_J^\infty)(\z,t) \equiv 0$. We fixed the known clinical parameters in 
agreement with the available literature of the field (see, 
e.g.,^^>\cite{Gatto20,Dimarco22} and references therein). In all subsequent 
figures, we highlighted the evolution of system^^>\eqref{eq:mass}--\eqref{eq:seirAme2} 
obtained in the deterministic cases $\delta \equiv -1$ or $\delta \equiv 1$. Also, we choose  the case $p = 1/2$, to show the performance 
in an intermediate case. 

As done in^^>\cite{Dimarco22}, we solved a least square problem to minimize the 
relative $L^2$ norm of the difference between the reported number of infected 
$\widehat{\rho_I}$ and recovered $\widehat{\rho_R}$, and the theoretical 
evolution of the model $\rho_I(t)$
and $\rho_R(t)$, with $t$ varying in the timespan^^>$[t_0, t_L]$ preceding the 
lockdown regime. For what concerns the initial data, we assume that 
$\rho_E(t_0) = \rho_I(t_0) = \rho_R(t_0) = 1$, i.e., $t_0$ marks nearly the 
start of the epidemics, while for the average initial number of contacts we set 
$m_S(t_0) = m_E(t_0) =  m_R(t_0) = 10$, in agreement with
the experimentally observed mean number of contacts in a Western country
before the pandemic^^>\cite{Beraud15}. In order to take into account illness 
and quarantine periods for infected individuals, we fixed their mean number of 
contacts to be $m_I(t) \equiv 3$ throughout their infection, which corresponds
to the average number of family contacts. Thus, the constrained minimization 
problem is the following:
\(
\min_{\beta,\lambda}
\left[
(1 - \theta)\norma{\rho_I(t) - \widehat{\rho_I}(t)}_{L^2([t_0,t_L])} + 
\theta\norma{\rho_R(t) - \widehat{\rho_R}(t)}_{L^2([t_0,t_L])}\right],
\label{eq:minproblem1}
\)
where $\theta \in [0, 1]$, $\norma{\,\cdot\,}_{L^2([t_0,t_L])}$ is the relative 
norm over the time horizon $[t_0,t_L]$, while we constrained $\beta$ to belong 
to the interval $[0, 0.01]$ and $\lambda$ to satisfy $3< \lambda \le 10$.
\begin{table}\centering
\begin{tabular}{lcccc}
\toprule
Parameters & $\zeta$ & $\gamma$ & $\beta$          & $\lambda$\\
\midrule
Values & $1/3.32$  & $1/10$     & $0.0176\div0.0226$   & $5$\\
\bottomrule
\end{tabular}
\caption{Parameters relative to system^^>\eqref{eq:mass}--\eqref{eq:seirAme2} and  obtained by solving problem^^>\eqref{eq:minproblem1} for $p \in [0,1]$, with constrains $\beta \in [0, 10^{-2}]$ and $\lambda \in (3, 10]$. In all tests we considered $\theta = 10^{-3}$, to better capture the trend for the infected cases.}
\label{tab:fittedparam}
\end{table}

In Table^^>\ref{tab:fittedparam} we report the parameters obtained by solving 
problem^^>\eqref{eq:minproblem1} for different choices of the parameter $p$, 
where we fixed the norm coefficient^^>$\theta=10^{-3}$, to better observe the 
trend with respect to the infectious individuals.
\begin{figure}
    \centering
    \includegraphics[width=0.5\linewidth]{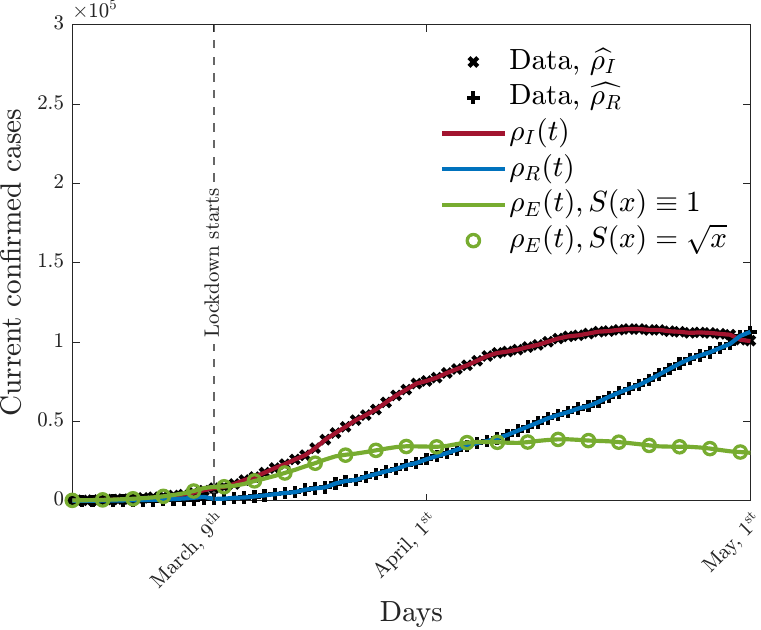}
    \caption{\small \textbf{Test 4a.} Comparison between data 
    relative to reported infected and recovered people (respectively, black 
    crosses and black plus signs) and time evolution of the mass fractions of 
    infectious agents^^>$\rho_I(t)$ (red solid line) and removed 
    agents^^>$\rho_R(t)$ (blue solid line), as prescribed by 
    system^^>\eqref{eq:mass}--\eqref{eq:seirAme2} with target $x_T(t)$ obtained 
    by solving problem^^>\eqref{eq:mincontrol}. We also reported the evolution 
    in time of the mass fraction of the exposed compartment^^>$\rho_E(t)$ 
    (green solid line and green circles). In all cases both selection functions 
    $S(x) \equiv 1$ and $S(x) = \sqrt{x}$ were employed, but we report distinct 
    curves only for the exposed agents for better clarity, since in all cases 
    we obtain nearly superimposable results, also with respect to different 
    choices of $p$. Epidemiological parameters as in 
    Table^^>\ref{tab:fittedparam}. }
	\label{fig:figure6}
\end{figure}

\subsubsection{Test 4b: Assessment of different restriction strategies}

Once all the epidemiological parameters are estimated, we can focus on the 
constrained regime, i.e., the
subsequent lockdown phase. Within the framework of our model, we can interpret 
the lockdowns enforced during the first wave of the pandemic in Western Europe 
as a form of control strategy whose associated selection function^^>$S(x)$ is 
uniform with respect to the number of contacts^^>$x$.
\begin{figure}
\includegraphics[width=0.5\linewidth]{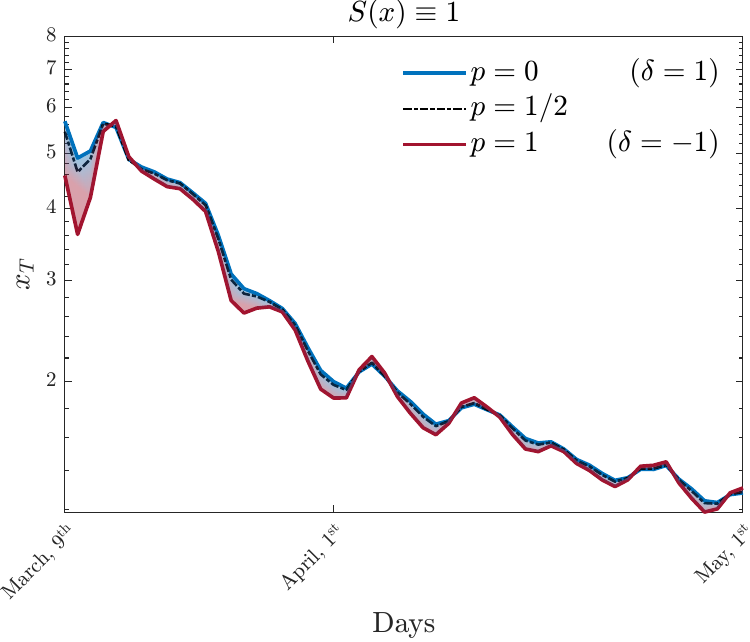}%
\includegraphics[width=0.5\linewidth]{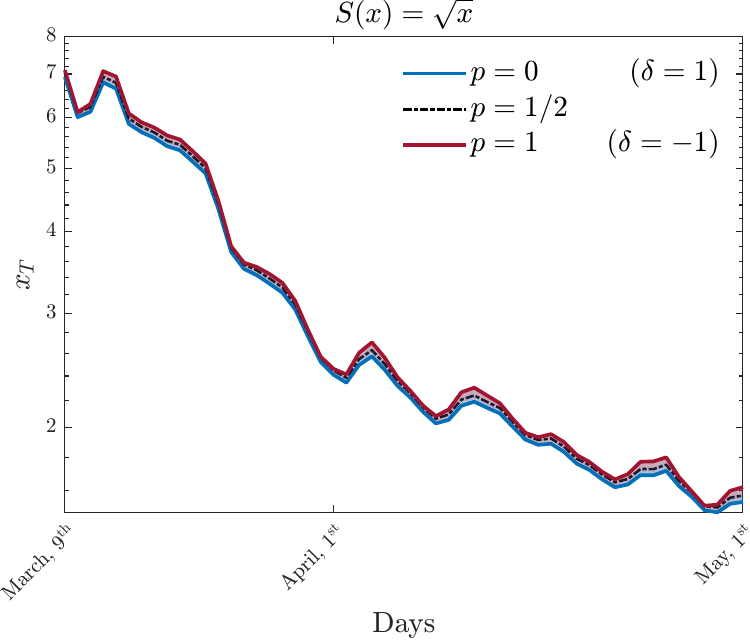}\\[2ex]
\includegraphics[width=0.5\linewidth]{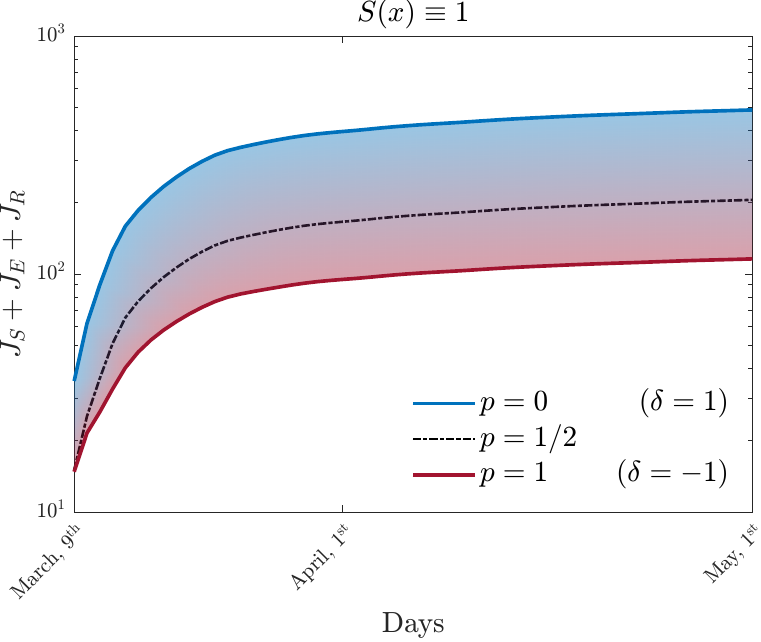}%
\includegraphics[width=0.5\linewidth]{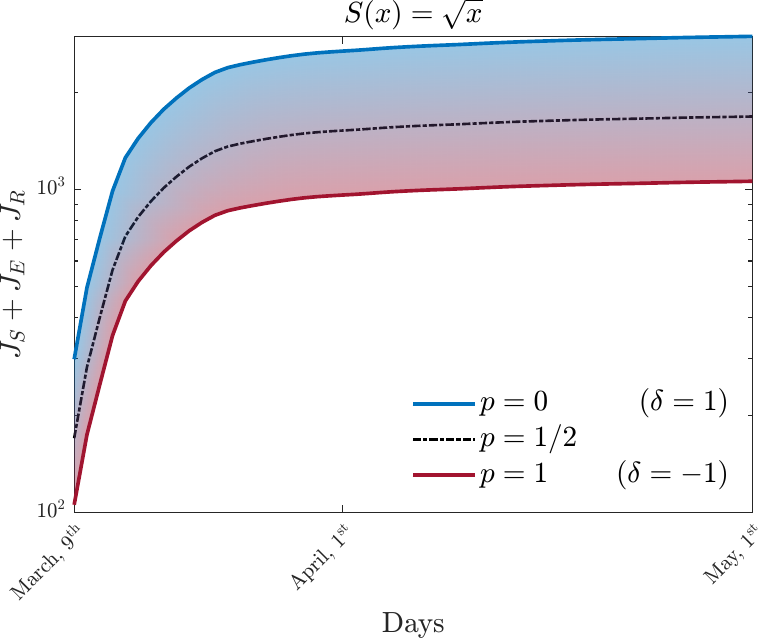}
    \caption{\small \textbf{Test 4b.} Time evolution in 
    semi-log scale of the targets $x_T(t)$ (top row) and associated total costs 
    (bottom row), obtained by solving problem^^>\eqref{eq:mintarget} and then 
    computing the cost with^^>\eqref{eq:Jcost}. On the left column, we have the 
    evolution in time of fitted targets and related cost for selection 
    function^^>$S(x) \equiv 1$; on the right column those relative to selection 
    function^^>$S(x) = \sqrt{x}$. We set $\delta(\z) = 1 - 2\z$, $\z \sim 
    \mathrm{Bernoulli}(p)$, with $p = 0$, $1/2$, and $1$. Epidemiological 
    parameters are reported in Table^^>\ref{tab:fittedparam}.}
	\label{fig:figure7}
\end{figure}
With this perspective, it is interesting to compute the optimal target value 
$x_T$ in the control term which permits to fit the data. As a simplifying 
assumption, we assume this value $x_T$ equal
for each compartment and we study the two cases of uniform and selective
restrictions, which can be obtained by fixing in the dynamics $S (x) = 1$ and 
$S (x) = \sqrt{x}$, respectively, while we can compute $G_J(f_J^\infty)(\z,t)$ 
by equation^^>\eqref{eq:CJfinfty}. Hence, we solve an optimization
problem in the lockdown timespan $[t_L + 1, t_f ]$, for a sequence of time
steps $t^n$ over a moving time window of one week (we tried to keep the 
notation consistent with the one in^^>\cite{Dimarco22}). Again, it is a 
constrained least-square problem:
\(
\min_{x_T(t^n)\in \R^+}
\left[
(1 - \theta)\norma{\rho_I (t) - \widehat{\rho_I} (t)}_{L^2([t^n-k_L,t^n + k_r 
])}
+\theta\norma{\rho_R(t) - \widehat{\rho_R}(t)}_{L^2([t^n-k_L,t^n + k_r ])}
\right],
\label{eq:mintarget}
\)
with $k_L = 3$, $k_r = 4$. We report the result of such fitting in 
Figure^^>\ref{fig:figure6},
along with the associated estimated 
evolution of the exposed compartment for both selection functions $S(x)\equiv 
1$ and $S(x) = \sqrt{x}$. In this case, we report the results only for the 
value $p = 1/2$, since the fitting procedure gives almost indistinguishable 
results with respect to the choice of^^>$p\in [0,1]$. 

We observe that the estimated value for the target^^>$x_T$ is higher when a 
selective lockdown is enforced, meaning that employing a non-uniform control 
strategy would achieve the same effects with respect to the number of infected 
people while allowing greater sociality, especially for the first part of the 
restriction period.

We also computed the total cost of such measures as the sum of the functionals 
$J_S + J_E + J_R$, where $J_H$ is defined for $H \in \I$ as
\(
J_H = \frac12 \intRp \lrp*{1 + \frac{S^2(x)}{\nu}}(x - x_T)^2 f_H^\infty(x)\, 
dx,
\label{eq:Jcost}
\)
that is, the functional^^>\eqref{eq:Jfunctional} can be seen as the 
instantaneous approximation of^^>$J_H$, which is obtained by 
considering^^>\eqref{eq:ustar} in the limit^^>$\epsilon,\tau \to 0^+$ 
(see^^>\cite{Dimarco22}). We see that the cost is strongly influenced by the 
considered selective strategy.
In Figure^^>\ref{fig:figure7}
we report both the estimated target values^^>$x_T$ (top row) and the associated 
total cost (bottom row) for both selection functions^^>$S(x) \equiv 1$ and 
$S(x) = \sqrt{x}$ for the 
choices of $p = 0$, $1/2$ and^^>$1$. 
\begin{figure}
    \centering
   \includegraphics[width=0.33\linewidth]{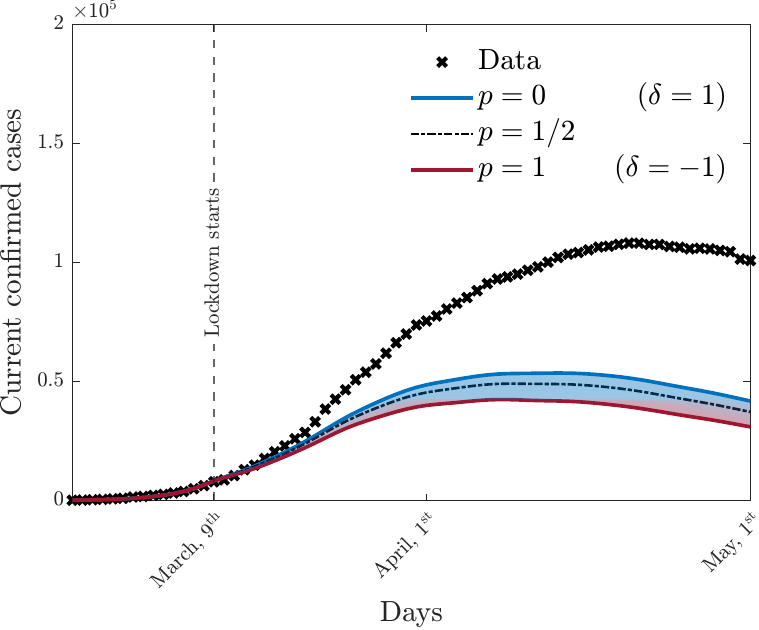}
	\includegraphics[width=0.33\linewidth]{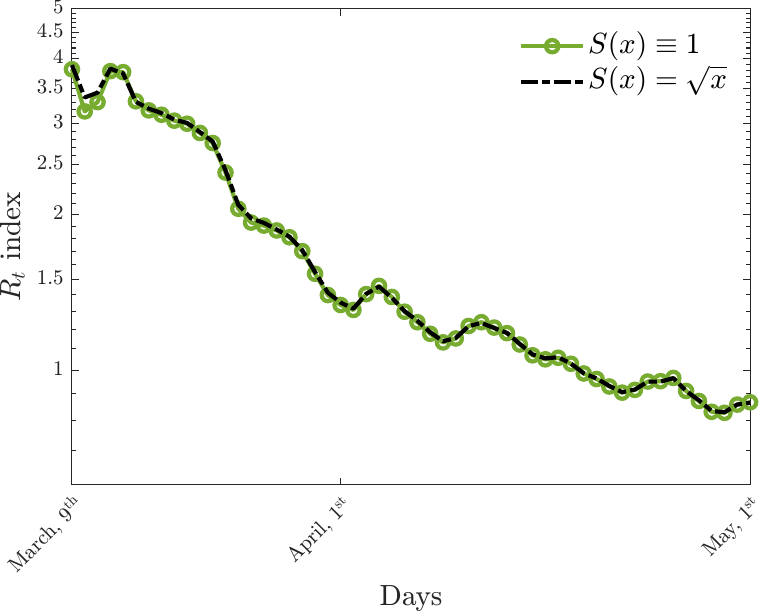}%
	\includegraphics[width=0.33\linewidth]{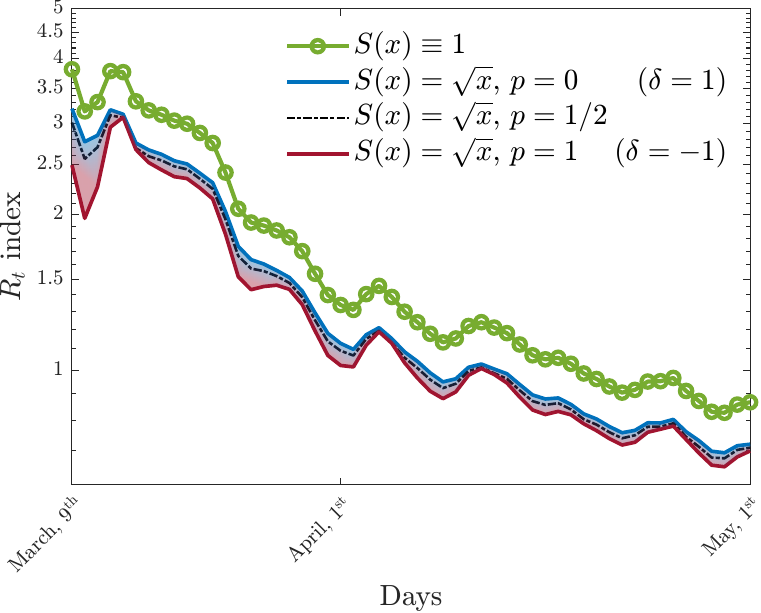}
    \caption{\small \textbf{Test 4b} Comparison 
    between data (black crosses), corresponding to a uniform control strategy 
    and time evolution of the mass fraction of infectious agents^^>$\rho_I(t)$ 
    (shaded area with solid borders in color) as prescribed by 
    system^^>\eqref{eq:mass}--\eqref{eq:seirAme2} with selection 
    function^^>$S(x) = \sqrt{x}$. We set $\delta(\z) = 1 - 2\z$, $\z \sim 
    \mathrm{Bernoulli}(p)$, with $p = 0$, $1/2$, and $1$. Epidemiological 
    parameters as in Table^^>\ref{tab:fittedparam}.}
	\label{fig:figure8}
\end{figure}

Finally, it is interesting to consider a retrospective analysis where the 
estimated $x_T$ associated to $S(x) \equiv 1$ is instead implemented in the 
dynamics with selective control. This means fixing the number of social 
contacts achievable with different selective functions and comparing the 
results on the evolution of the epidemics.
In Figure^^>\ref{fig:figure8}
we show 
the evolution of the disease in the presence of selective control with a target estimated by the uniform control. We observe that the peak of the epidemic is 
effectively reduced, suggesting that a selective
control strategy is an effective choice in fighting the spreading of the 
infection even in case of contact uncertainties. This 
extends the findings reported in^^>\cite{Dimarco22}, which proposed encouraging 
results in this way and a slim-tailed 
contact distribution.

\section*{Conclusion}
In this paper, we concentrated on the definition of non-pharmaceutical interventions in the presence of an uncertain contact distribution of the system of agents. To this end, we introduced a mathematical description of the epidemic by integrating an SEIR compartmental model with kinetic equations with uncertainties. Hence, we introduced a selective control strategy to force the number of contacts towards a fixed target. Observable effects of the control are then derived at the macroscopic level of description through classical methods of kinetic theory. Furthermore, we have proved that it is possible to reduce the variability of the mean number of connections and, therefore, to lower the impact of missing information on the system of agents. Possible extensions of the presented approach will concentrate on more sophisticated compartmentalizations.

\section*{Acknowledgements}
This work has been written within the activities of the GNFM group of INdAM (National Institute of High Mathematics). MZ acknowledges partial support of MUR-PRIN2020 Project No.2020JLWP23 (Integrated Mathematical Approaches to Socio-Epidemiological Dynamics). MZ and AM acknowledge the support of the Banff International Research Station (BIRS) for the Focused Research Group [22frg198] “Novel perspectives in kinetic equations for emerging phenomena”, July 17-24, 2022, where part of this work was done.

\bibliographystyle{abbrv}
\bibliography{SEIR_UQ.bib}

\end{document}